\documentclass[12pt,final]{amsart}
\usepackage[margin=1in]{geometry}
\usepackage{color}  
\usepackage{graphicx}
\usepackage{amssymb}
\usepackage{amsrefs}

 \theoremstyle{plain}
 \begingroup
 \newtheorem{thm}{Theorem}[section]
  
 \newtheorem{lem}[thm]{Lemma}
 \newtheorem{prop}[thm]{Proposition}
 
 \endgroup

 \theoremstyle{definition}
 \begingroup

 \endgroup

 \begingroup
 
 \endgroup

 \numberwithin{equation}{section}

\newcommand{\N}{\mathbb{N}}

\newcommand{\R}{\mathbb{R}}

\newcommand{\BUC}{{\rm BUC\,}}

\newcommand{\bc}{C_{\rm b}}

\newcommand{\al}{\alpha}

\newcommand{\ep}{\varepsilon}

\newcommand{\lam}{\lambda}
\newcommand{\sig}{\sigma}

\newcommand{\ol}{\overline}

\newcommand{\pl}{\partial}

\def\lan{\langle}
\def\ran{\rangle}

\def\limssup{\mathop{\rm limsup\!^*}}

\def\fr#1#2{\dfrac{#1}{#2}}

\def\XXint#1#2#3{{\setbox0=\hbox{$#1{#2#3}{\int}$}
     \vcenter{\hbox{$#2#3$}}\kern-.5\wd0}}

\def\bye{\end{document}} \def\by{\end{proof}\end{document}}

\def\R{{\mathbb R}}
\def\N{{\mathbb N}}

\def\mid{\,:\,}
\DeclareMathOperator{\dist}{dist}

\def\1{\mathbf{1}}

\def\gth{\theta}
\def\beq{\begin{equation}}
\def\eeq{\end{equation}}
\def\bproof{\begin{proof}}
\def\eproof{\end{proof}}
\def\bcases{\begin{cases}}
\def\ecases{\end{cases}}
\def\bprop{\begin{prop}}
\def\eprop{\end{prop}}
\def\blem{\begin{lem}}
\def\elem{\end{lem}}

\date{\today}

\def\ep{\varepsilon}
\def\gO{\varOmega}
\def\tim{\times}
\def\ga{\alpha}

\def\gd{\delta}
\def\gl{\lambda}
\def\pl{\partial}

\def\gs{\sigma}
\def\fr{\frac}

\def\ol{\overline}
\def\disp{\displaystyle}
\def\erf{\eqref}
\def\gL{\Lambda}
\begin{document}
\title[On the Langevin equation with variable friction]
{On the  Langevin equation with variable friction}
\author[H.~Ishii, P.~E.~Souganidis and H.~V.~Tran]{Hitoshi~Ishii$^{1,*}$, Panagiotis~E.~Souganidis${}^{2}$ 
and Hung~V.~Tran${}^{3}$}
\thanks{${}^*$ Corresponding author}
\thanks{${}^{1}$ Faculty of Education and Integrated Arts and Sciences, Waseda University, Nishi-Waseda, Shinjuku, Tokyo 169-8050, Japan. 
Partially supported by the KAKENHI No. 26220702,
and No. 16H03948, JSPS} 
\thanks{${}^{2}$ Department of Mathematics, The University of Chicago, 5734 S. University Avenue, Chicago, IL 60657, USA. 
Partially supported by the National Science Foundation Grants DMS-1266383 and DMS-1600129 and the Office for Naval Research Grant N00014-17-1-2095} 
\thanks{${}^3$ Department of Mathematics, 
The University of Wisconsin-Madison, 480 Lincoln Drive
Madison, WI 53706, USA. Partially supported by the  National Science Foundation Grants DMS-1615944 and  DMS-1664424. This work began while at The University of Chicago as Dickson Instructor. }

\markboth{Ishii, Souganidis and Tran}{Langevin equation with variable friction}
 \keywords{Langevin equation; Smoluchowski-Kramers approximation; 
asymptotic behavior;  diffusion processes; viscosity solutions}
\subjclass[2010]{
35B40, 
35B25, 
49L25
}
\email{hitoshi.ishii@waseda.jp (Hitoshi Ishii),
souganidis@math.uchicago.edu (Panagiotis E. Souganidis), 
hung@math.wisc.edu (Hung V. Tran)}

\begin{abstract}
We study two asymptotic problems for the Langevin equation with variable friction coefficient. The first is the small mass asymptotic behavior, known as the  Smoluchowski-Kramers approximation, 
of  the Langevin equation with strictly positive variable friction. 
The second result is about the limiting behavior of the solution when the friction vanishes in regions of the domain. 
Previous works on this subject  considered  one dimensional settings with the conclusions based on explicit computations.  
\end{abstract}

\maketitle



\section{Introduction}
\noindent The first topic in this paper 
is the  study of the small mass asymptotic behavior, known as  the  Smoluchowski-Kramers approximation, 
of the generalized Langevin equation. 
The latter describes the motion, with variable strictly positive friction coefficient  $\lam$, of a particle of mass $\mu$ in a force field $b$ which subject to random fluctuations modeled by a Brownian motion $W$ with diffusivity $\sigma$, which  represent random collisions of the given particle with other particles in the fluid. 
\smallskip

More precisely, we consider the behavior as $\mu \to 0$ of the solution ${x}^{\mu}$ to 
\begin{equation}\label{lv-1}
\mu \ddot{x}^{\mu} = b(x^{\mu})-\lam(x^{\mu}) \dot x^{\mu}
+\sig(x^{\mu}) \dot W,
\quad x^{\mu}(0)=x \in\R^n, \quad \dot x^{\mu}(0)=p \in \R^n. 
\end{equation}
The result is that, under some regularity assumptions on $b, \sigma$ and $\lam$,  and for every $T>0$ and $\delta >0$, 
\begin{equation}\label{takis1}
\lim_{\mu \to 0} \mathbb P(\max_{0\leq t\leq T}|x^{\mu}(t) -x (t)| >\delta)=0,
\end{equation}
with $x$ evolving by the It\^o  stochastic differential equation
\begin{equation}\label{takis2}
d x= \frac{1}{\lam (x)} b(x)dt + \frac{1}{\lam (x)}\sigma(x)d{W}, \quad x(0)=x \in \R^n.
\end{equation}

We prove \eqref{takis1} by studying the pde governing the law of $x^\mu$ and showing that, as $\mu \to 0$, its solutions converge 
to solutions to the pde of the law of $x$.  
\smallskip

To this end, we rewrite \eqref{lv-1} as 
\begin{equation}\label{lv-sys}
\begin{cases}
\dot x^\mu=y^\mu,\\
\dot y^\mu=\dfrac{1}{\mu}\left(b(x^\mu)-\lam(x^\mu)y^\al\right)+\dfrac{1}{\mu} \sig(x^\mu) \dot W. 
\end{cases}
\end{equation}
The  generator $L^\mu$ of the law of the diffusion process $(x^\mu,y^\mu)$ is, with $a=\sigma\sigma^t$, 
\begin{equation}\label{L}
L^\mu u(x,y):=\dfrac{1}{2\mu^2} a_{ij}(x)u_{y_i y_j} +\dfrac{1}{\mu}(b_i(x)-\lam(x)y_i)u_{y_i} + y_i u_{x_i}.
\end{equation}
The claim (see Theorem~2.1) is that solutions $u^\mu=u^\mu (x,y,t)$ to $u^\mu_t=L^\mu u^\mu$ converge, as $\mu \to 0$ and locally uniformly, to solutions $u=u(x,t)$ to $u_t=Lu$, where 
\begin{equation}\label{takis3}
Lu:=\dfrac{1}{2\lam(x)}a_{ij}(x)\left( \dfrac{u_{x_i}}{\lam(x)} \right)_{x_j}+\dfrac{1}{\lam(x)} b_i(x)u_{x_i}.
\end{equation}  
A result of this type  was shown by Freidlin and Hu \cite{FrH1} under some simplifying assumptions, for example $a\equiv 1$ 
by exact computations.  
\smallskip

The second topic of the paper is the study of the limiting generator at places where the friction vanishes. This question was raised 
by Freidlin, Hu and Wentzell \cite{FrHW}, who considered that problem in one dimension with $a\equiv 1$ and found an explicit solution.  Assuming  that the nonnegative friction vanishes in some compact 
region, \cite{FrHW} approximates $\lambda$ by  $\lambda + \ep$ and studies the behavior of the solutions as $\ep\to 0$.
\smallskip

Motivated by \cite{FrHW}, we consider the general boundary value problem 
\begin{equation}\label{takis4}  
-a_{ij}(x) \left ( \dfrac{u^\ep_{x_i}}{\lam+\ep} \right )_{x_j}-2b_i u^\ep_{x_i}= 0 \  \text{in} \  U, \quad  u^\ep=g  \ \text{on} \ \partial U,
\end{equation}
in a domain $U\subset \R^n$  (all the precise assumptions are stated later in the paper) and $\lambda \equiv 0$  in $V \subset U$ and strictly positive in $\ol U\setminus \ol V$.
\smallskip

The result (see Theorem~\ref{thm:conv}) is  that, as $\ep \to 0$ and uniformly in $\ol U$, $u^\ep \to u$, the unique viscosity solution to 
\beq\label{outside}
\begin{cases}
-a_{ij}(x) \left ( \dfrac{u_{x_i}}{\gl} \right )_{x_j}-2b_i u_{x_i}= 0 \  \text{in}\ U\setminus \ol V, \quad u=g \  \text{on} \ \partial U,\\[1mm]
-a_{ij}u_{x_i x_j}=0 \ \text{ in} \ V \ \ \ \text{ and }  \quad a_{ij}u_{x_i}\nu_j=0 \ \text{on} \ \partial V,\\[2mm]
{\displaystyle  \int_{\pl U}} \cfrac{a_{ij}u_{x_i}\nu_j m}{\gl}d\gs=0,
\end{cases}
\eeq
where $m\in C(\ol U)$ is the unique solution of an appropriate adjoint problem and $\nu$ denotes  the external normal vector to $V$ and $U$.
\smallskip

\subsection*{Organization of the paper} In the next section we introduce the precise assumptions and state the main results. Section~3 is devoted to the proof of the small mass approximation. In Section~4, we prove the result about the degenerate friction. 
In Section~5, we study the adjoint problems that play an  important role in the proofs in Section~4 and in identifying the limit. 
In Section~6, we give a brief explanation how to apply the standard theory of existence and uniqueness of solutions to the 
initial value problem for $u_t=L^\mu u$.  
\subsection*{Terminology and Notation} Depending on the context throughout the paper solutions are either classical or in the viscosity sense. In particular the boundary value  problem $$-a_{ij}u_{x_i x_j}=0 \ \text{ in} \ V \ \ \ \text{ and } \quad  a_{ij}u_{x_i}\nu_j=0 \ \text{on} \ \partial V,$$
is interpreted in the viscosity sense, that is, in the case of 
subsolution, for instance,  
\[-a_{ij}u_{x_i x_j}\leq 0 \ \text{ in} \ V \ \ \ \text{ and } \quad  
\min(-a_{ij}u_{x_i x_j},\,a_{ij}u_{x_i}\nu_j)\leq 0 \ \text{on} \ \partial V.
\] 
Given $O\subset \R^k$ for some $k$, $\BUC(O)$ is the space of bounded uniformly continuous functions on $O$; its norm is denoted by $\|\cdot\|$. Moreover, $\bc(O)$ and $\bc^m(O)$ with $m\in\N$ 
denote respectively the spaces of the bounded continuous functions on $O$ and 
the  functions in $\bc(O)$ with bounded continuous derivatives up to order $m$. Their respective  norms are $\|\cdot\|_{C(O)}$ and 
$ \|\cdot\|_{C^m(O)}$.  When possible, the dependence on the domain of the spaces in the norms is omitted. 
\smallskip

If $f_\ep: A\to \R$, where $A\subset \R^k$, is such that $\sup_{\ep\in (0,1)}\|f_\ep\|<\infty$,  the generalized 
(relaxed) upper and lower limits $f^+$ and $f^-$  are given respectively by
$$f^+(x):=\limssup_{\ep \to 0} f_\ep(x):=\limsup_{\ep \to 0, x' \to x} f_\ep(x') \ \text{and} \ f^-(x):=\liminf_{\ep \to 0}{\kern-3pt}_* f_\ep(x):=\liminf_{\ep \to 0, x' \to x} f_\ep(x').$$
Finally, $O(r)$ denotes various functions of $r\geq 0$ such that 
 $|O(r)|\leq Cr$ for all $r\ge 0$ for some  constant $C>0$ which is independent of the various parameters in  the specific context. 
\smallskip

Throughout the paper in writing equations we use the summation convention.  

\section{The assumptions and the results} 

\subsection*{Small mass approximation.} In the first part of the paper we assume that 
\begin{equation}\label{ass1}
\sigma, b, \lam, D\gl \ \text{ are bounded and Lipschitz continuous on }  \R^n,
\end{equation}
and there exist $\Theta \geq \theta >0$ such that,  
for all $x, \xi \in \R^n$, 
\beq\label{ass2}
\theta \leq \lam(x) \leq \Theta
\eeq
and 
\begin{equation}\label{ass3}
\theta |\xi|^2 \leq a_{ij}(x)\xi_i\xi_j\leq \Theta |\xi|^2. 
\end{equation}
We remark that we have not tried to optimize our assumptions and some of the results definitely hold with less regularity on the coefficients.
\smallskip

Fix $T>0$ and  given $u^\mu_0 \in \BUC(\R^n\times \R^n)$, 
let $u^\mu \in 
\bc(\R^n\times \R^n\times [0,T])$ be the unique (viscosity) solution to the initial value problem 
\beq\label{takis5}
\begin{cases}
u^\mu_t=\dfrac{1}{2\mu^2} a_{ij}(x)u^\mu_{y_i y_j}+\dfrac{1}{\mu}(b_i(x)-\lam(x)y_i)u_{y_i}^\mu + y_{i} u_{x_i}^\mu  \ \text{in} \ \R^n\times \R^n \times (0,T),\\[2mm]
u^\mu(\cdot, \cdot, 0)= u^\mu_0  \ \text{in} \  \R^n\times \R^n.
\end{cases}
\end{equation}
Because  the coefficients $\gl(x)y_i$ are not globally Lipschitz continuous on $\R^n\tim\R^n$, \erf{takis5} is a bit out of scope 
of the classical theory of viscosity solutions (see \cite{CIL}).  Nevertheless, 
in view of \eqref{ass1} and \erf{ass2},  
there exists 
a unique viscosity solution of \erf{takis5}.   
We  discuss this issue  briefly (see Theorem~\ref{ce}) in Section~6.
\smallskip

The result about the small mass approximation is stated next.

\begin{thm} \label{takis101} Suppose \eqref{ass1}, \eqref{ass2} and \eqref{ass3}. Assume that  $u^\mu_0 \in \BUC(\R^n\times \R^n)$ is such that $\lim_{\mu \to 0} \sup_{x,y \in \R^n}|u^\mu_0(x,y)-u_0(x)|=0.$ 
Then, as $\mu \to 0$ and locally uniformly on $\R^n\times \R^n \times [0,T)$, $u^\mu \to u \in \BUC(\R^n \times [0,T])$, where $u$ is the unique solution to 
\begin{equation}\label{lim-u}
\begin{cases}
u_t=\dfrac{1}{2\lam(x)}a_{ij}(x) \left( \dfrac{u_{x_i}}{\lam(x)}\right)_{x_j} +\dfrac{1}{\lam(x)} b_i(x)u_{x_i} \ \text{in } \R^n \times (0,T),\\[2mm]
u(\cdot, 0)=u_0 \ \text{in} \  \R^n.
\end{cases}
\end{equation}
\end{thm}

As stated in the introduction, a special case of Theorem~\ref{takis101} for $n=1$ was proved in \cite{FrH1}.

\subsection*{Vanishing friction} 
We formulate next the result about the vanishing friction. We assume that for some $\ga\in (0,\,1)$, 
\beq\label{ass4}
\begin{cases}
U  \text{ is  a $C^{2,\ga}$-bounded, connected, open subset of $\R^n$},\\[1mm] 
V \text{ is a $C^{2,\ga}$-connected open subset of $U$ such that $\ol V\subset U$,  } 
\end{cases}
\eeq 
\beq\label{ass5}
a, b, \lam \in C^{2,\al} (\bar U),
\eeq
there exist $\Theta, \theta >0$ such that,  
for all $x\in \ol U$ and $\xi \in \R^n$, 
\beq\label{ass6}
0\leq \lam(x) \leq \Theta,
\eeq
and
\begin{equation}\label{ass7}
\theta |\xi|^2 \leq a_{ij}(x)\xi_i\xi_j\leq \Theta |\xi|^2,
\eeq
\beq\label{ass7'}
\lambda \equiv 0 \ \text{on} \ \ol V 
 \  \ \text{ and } \ \ \gl>0  \ \text{ on }\ \ol U\setminus \ol V,
\eeq
and, if $d$ is the signed distance function of $\pl V$ given by
\[
d(x):=
\bcases
\dist(x,\pl V) \ \ \text{ if  } \ \ x\in\ol U\setminus V,\\[1mm]
-\dist(x,\pl V) \ \ \text{ if } \ \ x\in V,
\ecases
\]
then
\beq\label{ass8}\left\{\begin{minipage}{0.8\textwidth}
there exist $\gl_0\in C^2(\R)$ and $C_0>0$ such that
 $\gl_0\equiv 0$ on $(-\infty, 0],$ and 
$\gl_0(r)>0$,
$\gl_0' (r)\geq 0$ 
and   $r\gl_0'(r)\leq C_0\gl_0(r)$ for $r\in (0,\infty)$, and  
\[
\gl(x)=\gl_0(d(x)) \ \ \text{ in a neighborhood of $\pl V$ in $U\setminus \ol V$.}
\]
\end{minipage}
\right.
\eeq
Assumption \erf{ass8} is crucial 
in Lemmas \ref{l0.1} and \ref{l0.2} below. 
In what follows, one may replace $d$ by a defining function $\rho\in C^2(\R^n)$ of $V$, that is, $\rho \in C^2$ such that   $\rho<0$ in $V$,  $\rho>0$ in  $\R^n\setminus \ol V$,  and $D\rho \not=0$ 
on $\pl V$.  Finally, as before,  we remark that here we are not trying to optimize  the assumptions. 
\smallskip

We study the behavior, as $\ep \to 0$,  of the solution $u^\ep$ to \eqref{takis4} with  
\beq\label{ass9}
g \in C^{2,\al}(\ol U). 
\eeq

An important ingredient  of  our analysis  is the study of the asymptotic behavior of the solution $m^\ep$  of the ``adjoint'' problem
\beq \label{Ade} \tag*{(Ad)$_\ep$}
\bcases \disp
-\left(\fr{(a_{ij}m^\ep)_{x_i}}{\gl+\ep}-2b_j m^\ep\right)_{x_j}=0 \ \text{ in } \ U\\[3pt] \disp
\left(\fr{(a_{ij}m^\ep)_{x_i}}{\gl+\ep}-2b_j m^\ep\right)\nu_j=0 \  \text{ on }\pl U\\[3.5pt] \disp
\int_U m^\ep dx=1  \  \text{and}  \ m^\ep>0 \ \text{in} \ \ol U,
\ecases
\eeq
where 
$\nu$ denotes the outward unit normal vector to  $\ol U$. 
\smallskip

The limit problem of \ref{Ade}, as $\ep\to 0$, is 
\beq\label{Ad1}\tag{Ad1}
\bcases\disp
-\left(\fr{(a_{ij}m)_{x_i}}{\gl}-2b_j m\right)_{x_j}=0 \  \text{ in } \ U\setminus \ol V,&\\ \disp
\left(\fr{(a_{ij}m)_{x_i}}{\gl}-2b_j m\right)\nu_j=0 \  \text{ on } \ \pl U, &\\[3pt] \disp
\int_U m dx=1  \ \text{ and } \ m>0 \ \text{in} \ \bar U,
\ecases
\eeq  
and 
\beq\label{Ad2}\tag{Ad2}
-(a_{ij}m)_{x_ix_j}=0 \  \text{ in } V \ \ \ \text{ and }\quad (a_{ij}m)_{x_i}\nu_j=0 \ \text{ on }\ \pl V,
\eeq 
where, here,  $\nu$ is  the outward unit normal vector to  $\ol   V$. 
\smallskip

To describe the limiting behavior of the $u^\ep$'s we need the following result which is a consequence of Theorem \ref{thm:Ad} 
below whose proof is provided in Section~5.

\begin{thm} \label{thm:Ad12} Assume \eqref{ass4}, \eqref{ass5}, \eqref{ass6}, \eqref{ass7}, \erf{ass7'} and  \eqref{ass8}. Then there exists a unique solution $m\in C(\ol U)\cap C^2(\ol U\setminus \pl V)$
of \erf{Ad1} and \erf{Ad2}. 
\end{thm}
The main result is:

\begin{thm}\label{thm:conv} Assume  \eqref{ass4}, \eqref{ass5}, \eqref{ass6}, \eqref{ass7}, \erf{ass7'}, \eqref{ass8} and \eqref{ass9}. For each $\ep>0$ let $u^\ep\in C(\ol U)\cap C^2(U)$ be the unique solution to  \eqref{takis4}. Then, as $\ep \to 0$ and uniformly on $\ol U$, $u^\ep \to u,$ where $u\in C(\ol U)\cap C^2(\ol U\setminus \pl V)$ is the unique solution to 
\beq\label{outside'}
-a_{ij}(x)\left ( \dfrac{u_{x_i}}{\gl} \right )_{x_j}+2b_i u_{x_i}= 0 \  \text{in}\ U\setminus \ol V, 
\eeq
\beq\label{dc}
u=g \  \text{on} \ \partial U, 
\eeq
\beq\label{inside}
-a_{ij}u_{x_ix_j}= 0 \   \text{in} \ V 
\ \ \text{ and } \ \  a_{ij}u_{x_i}\nu_j=0 \  \text{on} \ \partial V,
\eeq
and 
\beq\label{balance}
\int_{\pl U} \fr{a_{ij}u_{x_i}\nu_j m}{\gl}d\gs=0,
\eeq
with  $m\in C(\ol U)$ is  given by Theorem \ref{thm:Ad12}.
\end{thm}
The meaning of \eqref{inside} was discussed in the subsection about terminology and notation earlier in the paper. 


\section{The small mass approximation}%

The proof of Theorem~\ref{takis101} is based on a variant of the perturbed test function (see Evans \cites{E1,E2}) and classical 
arguments from the theory of viscosity solutions. 
\subsection*{Formal expansion}
To identify the equation satisfied by the limit of the $u^\mu$'s we postulate the ansatz
\begin{equation}\label{ansatz}
u^\mu(x,y,t)=u(x,t)+\dfrac{\mu y_i}{\lam(x)} v_i(x,t)+\dfrac{\mu^2 y_i y_j}{\lam(x)^2} w_{ij}(x,t)+\dfrac{\mu^3  y_i y_j y_k}{\lam(x)^3} z_{ijk}(x,t)+\cdots
\end{equation}
where $u, v_i, w_{ij}, z_{ijk},\ldots$ are real-valued functions on $\R^n \times [0,\infty)$. 
\smallskip

We assume  that,  for  $1 \le i,j,k,\ldots \le n$, $w_{ij}=w_{ji}$, $z_{ijk}=z_{jik}=\ldots,$ 
we insert  \eqref{ansatz} in \eqref{takis5}, we organize in terms of powers of $\mu$  and we equate to $0$ the coefficients  of $O(1)$ and $O(\mu).$ 
\smallskip

From the former we get 
\beq\label{takis20}
u_t=\dfrac{1}{\lam^2} a_{ij} w_{ij}+\dfrac{1}{\lam}(b_i-\lam y_i)v_i+y_i u_{x_i}
=\dfrac{1}{\lam^2} a_{ij} w_{ij}+\dfrac{1}{\lam}b_i v_i+y_i (u_{x_i}-v_i),
\eeq
while from the latter we find 
\beq\label{takis21}
\dfrac{y_i}{\lam}v_{i,t}=\dfrac{3 y_i}{\lam^3} a_{jk} z_{ijk}+\dfrac{1}{\lam^2}(b^i-\lam y_i)2y_jw_{ij}
+y_i y_j \left ( \dfrac{v_i}{\lam} \right )_{x_j}.
\eeq
We deduce from \eqref{takis20} that $v_i=u_{x_i}$ for $1 \le i \le n$. Then \eqref{takis21}   can be written , for $1 \le i \le n$, as
$$
v_{i,t}=\dfrac{3}{\lam^2} a_{jk} z_{ijk}+\dfrac{2b_j}{\lam}w_{ij}+ y_j \left (  \lam \left ( \dfrac{u_{x_i}}{\lam}\right)_{x_j}-2 w_{ij}\right ),
$$
which yields that $$w_{ij}=2^{-1}\lam \left (\lam^{-1} u_{x_i} \right)_{x_j}.$$ 
Hence we obtain formally that $u=u(x,t)$ satisfies
\begin{equation}\label{conv-1}
u_t=\dfrac{1}{2\lam(x)}a_{ij}(x) \left ( \dfrac{u_{x_i}}{\lam(x)} \right)_{x_j}+\dfrac{1}{\lam(x)} b_i(x)u_{x_i}.
\end{equation}

\subsection*{The rigorous convergence} We present here the rigorous proof of the asymptotics.

\begin{proof}[Proof of Theorem \ref{takis101}] 
Fix $T>0$ and, without loss of generality, we only consider $\mu\in (0,1).$  
\vskip.035in
In view of our assumptions, the $u^\mu$'s are bounded on $\R^n\times \R^n \times [0,T]$ uniformly  in $\mu$. To deal with the special (unbounded) dependence of \eqref{takis5} on $y$, we find it necessary 
to modify the definition of the relaxed upper and lower limits, which was introduced earlier.
\smallskip

In particular, taking into account the estimate \eqref{conv-3} 
on $y^\mu$ below, 
we define generalized upper and lower limits $u^+$ and $u^-$  
on $\R^n\tim [0,\,T]$ by
\[
u^+(x,t)=\lim_{\gd\to 0+}\sup\{u^\mu(p,q,s)\mid 0<\mu<\gd, 
|p-x|<\gd, |s-t|<\gd, |\mu q|<\gd\},
\]
and 
\[
u^-(x,t)=\lim_{\gd\to 0+}\inf\{u^\mu(p,q,s)\mid 0<\mu<\gd, 
|p-x|<\gd, |s-t|<\gd, |\mu q|<\gd\},
\]
and prove that they are respectively sub- and super-solutions to \eqref{lim-u}. Since the arguments are almost identical, here we show the details only for the the generalized upper limit. Once the sub- and super-solution properties are established, we conclude, using that  \eqref{lim-u} has a comparison principle, that $u^+=u^-.$ This is a classical result  in the theory of viscosity solutions, hence we omit the  details.
\smallskip


We now show that $u^+$ is a viscosity subsolution to (2.5) 
on $\R^n\tim[0,\,T)$, that is, including  $t=0$.   
To this end, we assume that, for some smooth test function $\phi$, 
${u}^+-\phi$ has a strict global maximum at $(x_0,t_0) \in \R^n \times [0,T)$. 
\smallskip

For the arguments below, it is convenient to assume that, there exists  a compact neighborhood $N$ of $(x_0, t_0)$ such that 
\begin{equation}\label{takis60}
\begin{cases}
\text{$\phi$ is constant for all $(x,t)\in (\R^n\tim [0,T])\setminus N$}, \\[1.5mm]
 \inf_{(\R^n\tim [0,\,T))\setminus N} \phi>2\sup_{0<\mu<1}\|u^\mu\| \ \text{and } \ \phi(x_0,t_0)=0.
\end{cases}
\end{equation}

\smallskip

We use a perturbed test function type argument to show that,  at $(x_0,t_0)$, 
if $t_0>0$ or if $t_0=0$ and $u^+(x_0,0)>u_0(x_0)$, then  
\begin{equation}\label{conv-2}
\phi_t\le \dfrac{1}{2\lam}a_{ij} \left ( \dfrac{\phi_{x_i}}{\lam} \right)_{x_j}+\dfrac{1}{\lam} b_i\phi_{x_i}.
\end{equation}
First we consider the case $t_0>0$, in which case  we choose  $N$ so that $N\subset \R^n\tim(0,\,T)$. 
\smallskip

We fix some $K>0$ and  replace $\phi$ by 
$$
\psi(x,y,t)=\psi^\mu(x,y,t):=\phi+\dfrac{\mu}{\lam}y_i \phi_{x_i}(x,t)+\dfrac{\mu^2}{2\lam}y_i y_j \left (\dfrac{\phi_{x_i}}{\lam} \right )_{x_j} +K|\mu y|^3.
$$
Straightforward computations together with \eqref{ass2} give 
\[
\psi_t(x,y,t)=\phi_t(x,t)+O(|\mu y|+|\mu y|^2),
\]
\[\begin{aligned}
\fr{a_{ij}\psi_{y_i y_j}}{2\mu^2}
&\,=\fr{1}{2 }a_{ij}\left(\fr{1}{\lam}\left(\fr{\phi_{x_i}}{\lam}\right)_{x_j}
+3K\mu (|y|^{-1}y_i y_j +|y|\delta_{ij})\right)
\\&\,=\fr{1}{2\lam}a_{ij} \left(\fr{\phi_{x_i}}{\lam}\right)_{x_j}+O(K|\mu y|),\\[2mm]
\end{aligned}
\]
\[\begin{aligned}
\fr{b_i-\lam y_i}{\mu}\cdot \psi_{y_i}
&=\fr{b_i-\lam y_i}{\mu}\cdot\left(
\mu\fr{\phi_{x_i}}{\lam}+\fr{\mu^2}{2\lam} \left(y_j \fr{\phi_{x_i}}{\lam}
\right)_{x_j} + \fr{\mu^2}{2\lam} \left(y_j \fr{\phi_{x_j}}{\lam}
\right)_{x_i}
+3\mu^3 K|y|y_i\right)
\\&\,=\fr{b_j \phi_{x_i}}{\lam}
-y_j\ \left(\phi_{x_j}
+\mu \left(y_i \fr{\phi_{x_i}}{2\lam}\right)_{x_j}\right)\\&
-3K\mu^2\lam |y|^3
+O\left(\mu|y|+K|\mu y|^2\right),
\end{aligned}
\]
\and
\[\begin{aligned}
y_i\psi&_{x_i}\,=y_i \left(\phi_{x_i}
+\mu \left(y_j \fr{\phi_{x_j}}{\lam}\right)_{x_i}
+\fr{\mu^2}{2\lam}\left(y_ky_l\left(\fr{\phi_{x_k}}{\lam}\right)_{x_l}\right)_{x_i}
\right)
\\&\,=y_i \left(\phi_{x_i}
+\mu \left(y_j\fr{\phi_{x_j}}{\lam}\right)_{x_i}\right)+O(\mu^2 |y|^3);
\end{aligned}
\]
here, $O(r)$ is independent of $y$, $\mu$ and $K$. 
\smallskip

Combining the above we get, for some  $M>0$ depending only on $\phi$, $b$ and $\gl$, 
\[\begin{aligned}
   -\psi_t&+\fr{ a_{ij}\psi_{y_iy_j}}{2\mu^2}
+\fr{b_i-\lam y_i}{\mu}\psi_{y_i}
+y_i \psi_{x_i} \\
& \leq  -\phi_t+\fr{1}{2\lam}a_{ij} \left(\fr{\phi_{x_j}}{\lam}\right)_{x_i}
+\fr{b_i\phi_{x_i}}{\lam}
-3K\mu^2\lam |y|^3 +M((K+1)(|\mu y|+|\mu y|^2)+\mu^2|y|^2).
\end{aligned}
\]
Next we observe that $u^\mu-\psi$ has a global maximum on $\R^n\tim\R^n\tim(0,T)$. Indeed, note first that there exists  a constant $R=R^\mu>0$ such that
\begin{equation}\label{takis61}
\inf\{\psi(x,y,t):(x,y,t)\in \R^n\tim\R^n\tim[0,T),\ |y|>R \}
\geq 1+2\|u^\mu\|,
\end{equation}
consider the compact subset of $\R^n\tim\R^n\tim(0,T)$
\[
N^\mu:=\{(x,y,t)\mid (x,t)\in N,\ y\in \ol B_R\},
\]
note that 
\[
\psi(x,y,t)\leq \phi(x,t) \ \ \text{ if } \ (x,t) \in \R^n\tim(0,T)\setminus N,
\]
and, in view of \eqref{takis60} and \eqref{takis61}, if  $(x,y,t)\in\R^n\tim\R^n\tim(0,T)$ and $(x,t)\not\in N$, then  
\[\begin{aligned}
(u^\mu-\psi)(x,y,t) 
&\leq u^\mu(x,y,t)-\phi(x,t)
\leq \|u^\mu\|-\inf_{(\R^n\tim(0,\,T))\setminus N} \phi
\\&-1-\|u^\mu\|\leq -1+u^\mu(x_0,0,t_0)
=-1+u^\mu(x_0,0,t_0)-\phi(x_0,t_0) 
\\&=-1+u^\mu(x_0,0,t_0)-\psi(x_0,0,t_0),
\end{aligned}
\]

and, if $|y|>R$, then 
\[
\begin{aligned}
(u^\mu-\psi)&(x,y,t) 
\\&\leq \|u^\mu\|-\inf\{\psi(p,q,s)\mid (p,q,s)\in \R^n\tim\R^n\tim[0,T),\ |q|>R \}
\\&\leq -1+\|u^\mu\|
\leq -1+u^\mu(x_0,0,t_0)-\psi(x_0,0,t_0).
\end{aligned}
\]
The two inequalities above yield 
\[
\sup_{(\R^n\tim\R^n\tim(0,T))\setminus N^\mu}(u^\mu-\psi)
\leq -1+(u^\mu-\psi)(x_0,0,t_0)<\max_{N^\mu}(u^\mu-\psi),
\]
that is, $u^\mu-\psi$ has a global maximum at some point in
$N^\mu$. 
\smallskip

Let 
$(x^\mu,y^\mu,t^\mu)$ be a global  maximum point of $u^\mu-\psi$.  Then,  at $(x^\mu,y^\mu,t^\mu)$,
\[
-\psi_t+\fr{ a_{ij}\psi_{y_iy_j}}{2\mu^2}
+\fr{b_i-\lam y_i}{\mu}\psi_{y_i}
+y_i \psi_{x_i} \geq 0,
\]
and, hence, always at $(x^\mu,y^\mu,t^\mu)$, 
\[\begin{aligned}
-\phi_t+\fr{1}{2\lam}a_{ij} \left(\fr{\phi_{x_j}}{\lam}\right)_{x_i}
+\fr{b_i\phi_{x_i}}{\lam}
-3K\mu^2\lam |y|^3 \geq M((K+1)(|\mu y|+|\mu y|^2)+\mu^2|y|^3).
\end{aligned}
\]
Choosing $K=\frac{M+1}{3\theta}$ we obtain 
\begin{equation}\label{sub-i}
-\phi_t+\fr{1}{2\lam}a_{ij} \left(\fr{\phi_{x_j}}{\lam}\right)_{x_i}
+\fr{b_i\phi_{x_i}}{\lam}
\geq -M(K+1)(|\mu y|+|\mu y|^2)+\mu^2|y|^3.
\end{equation}
In particular, for some $C>0$, we find 
\[
\mu^2|y^\mu|^3\leq C(1+|\mu y^\mu |+|\mu y^\mu|^2).
\]
Hence, $|\mu y^\mu| =O(\mu^{1/3})$, and, thus, 
\begin{equation}\label{conv-3}
\lim_{\mu\to 0}\mu y^\mu=0 \ \  \text{and} \ \ \lim_{\mu\to 0}\left(\psi(x^\mu,y^\mu,t^\mu)-\phi(x^\mu,t^\mu)\right)=0. 
\end{equation}
Next we show that there is a sequence $\mu_j \to 0$ such that 
\[
\lim_{j\to \infty}(x^{\mu_j},t^{\mu_j})=(x_0,t_0). 
\]
In view of the definition of $u^+$, 
we may select   a sequence $\{(\mu_j,p_j,q_j,s_j)\}_{j\in\N}
\subset (0,\,1)\tim\R^n\tim\R^n\tim(0,T)$ such that 
\begin{equation} \label{takis62} 
\lim_{j\to\infty}(\mu_j,p_j,s_j)=(0,x_0,t_0),   \ 
 \lim_{j\to \infty}\mu_j q_j=0 \ \text{and} \ \lim_{j\to\infty}u^{\mu_j}(p_j,q_j,s_j)=u^+(x_0,t_0).
\end{equation}
Passing to a subsequence, we may assume that, for some $(\bar x,\bar t)\in N$,
\[
\lim_{j\to\infty}(x^{\mu_j},t^{\mu_j})=(\bar x,\bar t). 
\]
Since $(x^\mu,y^\mu,t^\mu)$ is a global maximum of 
$u^\mu-\psi$, for any $\gd>0$ and  as soon as $|x^{\mu_j}-\bar x|<\gd$, 
$|t^{\mu_j}-\bar t|<\gd$ and $|\mu_j y^{\mu_j}|<\gd$, we have
\begin{equation}\label{takis64} 
(u^{\mu_j}-\psi)(p_j,q_j,s_j)\leq 
(u^{\mu_j}-\psi)(x^{\mu_{j}},y^{\mu_j},t^{\mu_j})
\leq v^\gd(\bar x,\bar t)-\psi(x^{\mu_{j}},y^{\mu_j},t^{\mu_j}) 
\end{equation}
where $v^\gd$ is defined by 
\[
v^\gd(x,t):=\sup\{u^\mu(p,q,s)\mid |p-x|<\gd, |s-t|<\gd, |\mu q|<\gd\}.
\]
Now, since 
\[
\lim_{j\to \infty}\psi^{\mu_j}(p_j,q_j,s_j)=\phi(x_0,t_0) 
\ \ \text{ and } \ \ \lim_{j\to\infty}\psi^{\mu_j}(x^{\mu_j},y^{\mu_j},t^{\mu_j})=\phi(\bar x,\bar t), 
\]
we find from \eqref{takis64} that, for any $\gd>0$,
\[
(u^+-\phi)(x_0,t_0)\leq (v^\gd-\phi)(\bar x,\bar t),
\]
which readily gives 
\[
(u^+-\phi)(x_0,t_0)\leq (u^+-\phi)(\bar x,\bar t). 
\]
Since $(x_0,t_0)$ is a strict global maximum point of $u^+-\phi$, 
we see from the above that $(\bar x,\bar t)=(x_0,t_0)$,
that is, 
\[
\lim_{j\to \infty}(x^{\mu_j},t^{\mu_j})=(x_0,t_0). 
\] 
It then follows  from \eqref{sub-i} that, at $(x_0,t_0)$, 
\[
\phi_t
\leq 
\fr{1}{2\lam}a_{ij} \left(\fr{\phi_{x_i}}{\lam}\right)_{x_j}
+\fr{b_i \phi_{x_i}}{\lam}. 
\]

Now we consider the case  $t_0=0$ and $u^+(x_0,0)>u_0(x_0),$
and show that 
\eqref{conv-2} 
holds at $(x_0,0)$. 
\vskip.035in
Let  $\gd>0$ be such  that
\begin{equation}\label{takis65} 
(u^+-\phi)(x_0,0)>3\gd+ u_0(x_0)-\phi(x_0,0)
\end{equation}
and observe that there is a $\mu_0\in (0,\,1)$ 
such that
\begin{equation}\label{takis66}
\sup_{0<\mu<\mu_0}\|u_0-u^{\mu}_0\|<\gd.
\end{equation}
Fix such a $\mu_0$ and, henceforth, assume that $\mu\in (0,\mu_0).$   
Moreover, since in the definition of $\psi$, 
we have,  for some $C>0$ independent of $\mu$,
\[
\psi^\mu(x,y,t)\geq \phi(x,t)-C(|\mu y|+|\mu y|^2)+K|\mu y|^3,
\]
we may assume, 
choosing $K$ large enough independently of $\mu$, that 
\begin{equation}\label{takis67}
\psi^\mu(x,y,t)>\phi(x,t)-\gd \  \text{ for } \ (x,y,t)\in\R^n\tim\R^n
\tim(0,T). 
\end{equation}
Then we select $N$ to be  a compact neighborhood
of $(x_0,0)$ relative to $\R^n\tim [0,\,T)$ as before, with the  additional 
requirement, in view of \eqref{takis65}, that 
\begin{equation}\label{takis68}
(u^+-\phi)(x_0,0)>3\gd+u_0(x)-\phi(x,0) \  \text{ if } \ (x,0)\in N. 
\end{equation}

As before, we can select 
a global maximum point $(x^\mu,y^\mu,t^\mu)$ 
of $u^\mu-\psi$, where $(x^\mu,t^\mu)\in N$ 
for every $\mu\in(0,\mu_0)$, 
and  a sequence $\{(\mu_j,p_j,q_j,s_j)\}_{j\in\N}\subset 
(0,\,\mu_0)\tim \R^n\tim\R^n\tim[0,\,T)$ satisfying \eqref{takis62}.
Finally,  may assume that $(p_j,s_j)\in N$ for $j\in\N$. 
\smallskip

We claim that the sequence $\{t^{\mu_j}\}_{j\in \N}$ 
contains a subsequence, which we denote the same way as the sequence,  such that $t^{\mu_j}>0$. 
\smallskip

Indeed arguing by contradiction, we suppose that,  for $j\in\N$ large enough, 
$t^{\mu_j}=0$. 
Fix such $j\in\N$  and   
observe that
\[
(u^{\mu_j}-\psi)(x^{\mu_j},y^{\mu_j},0)
\geq (u^{\mu_j}-\psi)(p_{j},q_j,0),
\]
and, in view of \eqref{takis66}, \eqref{takis67} and   \eqref{takis68}, 
\[\begin{aligned}
(u^{\mu_j}-\psi)(x^{\mu_j},y^{\mu_j},0)
&<u^{\mu_j}_0(x^{\mu_j},y^{\mu_j},0)-\phi(x^{\mu_j},0)+\gd
<2\gd+u_0(x^{\mu_j})-\phi(x^{\mu_j},0)
\\&<-\gd +(u^+-\phi)(x_0,0).
\end{aligned}
\]
Hence, for such large $j$, we have
\[
(u^{\mu_j}-\psi)(p_{j},q_j,0) <-\gd+(u^+-\phi)(x_0,0).
\]
Letting $j\to \infty$ yields
\[
(u^+-\phi)(x_0,0)\leq-\gd+(u^+-\phi)(x_0,0),
\] 
which is a contradiction, proving the claim. 
\smallskip

We may now assume that $t^{\mu_j}>0$, for all large $j$,  and argue exactly as in the case $t_0>0$, to conclude  that \erf{conv-2} 
holds.  
\smallskip

This completes the proof of the subsolution property. 
\smallskip
 
It is well-known that if $u^+$ (resp. $u^-$) is a subsolution 
(resp. supersolution)  of
\erf{lim-u} in the viscosity sense, as in the proof above, then 
$u^+(x,0)\leq u_0(x)$ (resp. $u^-(x,0)\geq u_0(x)$) for all 
$x\in\R^n$. 
\end{proof}

\section{Vanishing variable friction} 

The following two results are important for the proof of Theorem~\ref{thm:conv}. The first asserts the existence of a uniques solution to adjoint problem. Its assertion (iii) is exactly Theorem \ref{thm:Ad12}. 
Its proof, which is rather long,  is presented in Section~5. 

\begin{thm} \label{thm:Ad} Assume \eqref{ass4}, \eqref{ass5}, \eqref{ass6}, \eqref{ass7}, \erf{ass7'} and  \eqref{ass8}. Then:

\emph{(i)} For any $\ep\in (0,\,1)$ there exists a unique 
solution $m^\ep\in C^2(\ol U)$ of \emph{\ref{Ade}}.

 \emph{(ii)} 
The family $\{m^\ep\}_{\ep\in(0,\,1)}$ converges, as $\ep\to 0$ and  uniformly on $\ol U$, to\\  $m\in C(\ol U)\cap C^2(\ol U\setminus\pl V).$ 

 \emph{(iii)} The function $m$ is the  unique solution  to 
\erf{Ad1}--\erf{Ad2}.
\end{thm}

Obviously, Theorem~\ref{thm:Ad12} is a direct consequence 
of Theorem~\ref{thm:Ad} above. 

The second preliminary result, which is proved at the end of this section, is about the behavior of the generalized upper-and lower limits $u^+$ and $u^-$ of the family $\{u^\ep\}_{\ep\in (0,\,1)}$ in $\ol U$.

\begin{lem}\label{lem:E}  Suppose the assumptions of Theorem \ref{thm:conv}.  Then

 \emph{(i)} the family $\{u^\ep\}_{\ep\in (0,\,1)}$ is 
uniformly bounded on $\ol U$, and  

\emph{(ii)} 
 $u^+$ and $u^-$ are respectively sub- and super-solution to \eqref{inside} in $\ol V$. 
\end{lem}

Accepting Theorem~\ref{thm:Ad} and Lemma~\ref{lem:E}, we complete the proof of Theorem \ref{thm:conv}.
\smallskip


Before presenting the proof  we recall Green's formula that we will use in several occasions below. For any $\phi,\psi\in C^2(\ol U)$ and $\ep>0$, we have:
\beq\label{vvf-1}
\begin{aligned}
\int_U& \left(a_{ij}\left(\fr{\phi_{x_i}}{\gl+\ep}\right)_{x_j}+2b_i \phi_{x_i} \right)\psi dx
\\&=\int_{\pl U} \left\{\fr{a_{ij}\phi_{x_i}\nu_j\psi}{\gl+\ep}
+\left(2b_i \psi-\fr{(a_{ij}\psi)_{x_j}}{\gl+\ep}\right)\nu_i\phi
\right\}d\gs
\\&\quad 
+\int_{U} \left(\left(\fr{(a_{ij}\psi)_{x_j}}{\gl+\ep}\right)_{x_i}
-2(b_j \psi)_{x_j}\right)\phi dx.
\end{aligned}
\eeq


\bproof [Proof of Theorem \ref{thm:conv}] For $\ep\in (0,\,1)$, let $m^\ep\in C^2(\ol U)$ be the unique 
solution of \ref{Ade}. 
\smallskip

Applying  \erf{vvf-1}  to $\phi=u^\ep$ and $\psi=m^\ep$, we get, for all $\ep\in (0,1)$,  
\beq\label{vvf-2}
\int_{\pl U}\fr{a_{ij}u^\ep_{x_i}\nu_j m^\ep}{\gl+\ep}d\gs=0. 
\eeq 
Theorem \ref{thm:Ad} yields  a unique  function $m\in C(\ol U)$ 
such that, as $\ep \to 0$, 
\beq\label{vvf-3}
m^\ep \to m \ \ \text{ uniformly on }\ol U  \ \text{and} \  m>0 \ \text{in} \ \ol U.
\eeq
Let $K$ be a compact  subset of $\ol U\setminus \pl V$. 
Since  $a_{ij}u^\ep_{x_ix_j}+2\ep b_iu_{x_i}^\ep=0$ in $V$,  the
assumption on $\lam$ implies that, for some  $c_K>0$,  
$\lam \geq c_K >0$ in $K\setminus V$, the classical  Schauder estimates (\cite{GiTr}*{Lemma 6.5}) yield 
a constant $C_K>0$ such that  
\beq\label{Schauder1}
\|u^\ep\|_{C^{2,\ga}(K)}\leq C_K. 
\eeq
It follows that there exist a sequence $\ep_k \to 0$ 
and $u_0\in C^2(\ol U\setminus \pl V)$ such that, as  $k\to\infty, $
\beq\label{conv-1'}
u^{\ep_k} \to u_0 \ \ \ \text{ in } C^{2}(\ol U\setminus \pl V). 
\eeq 
Let $u^+$ and $u^-$ be the relaxed upper and lower limits of $\{u^{\ep_k}\}_{k\in\N}$. 
It follows from  \erf{conv-1'} that 
\beq\label{takis300}
u^+=u^-=u_0 \ \text{ on } \  \ol U\setminus \pl V.\eeq
Moreover,  Lemma \ref{lem:E}  yields that $u^+$ and $u^-$ are respectively a viscosity sub- and super-solution 
to  \erf{inside}. Since any constant function is a solution to \erf{inside}, combining the strong maximum
principle as well as Hopf's lemma  we get  (see also 
Patrizi \cite{Patrizi}) 
that  
$$u^+= \max_{\ol V}u^+ \ \text{ and} \  u^-= \min_{\ol V}u^- \ \text{  on} \  \ol V.$$  
Then \eqref{takis300} gives 
that 
\[
u^+=u^- \ \ \text{ on }\ol V,
\]
which proves that 
$u^+=u^-$ on $\ol U$. 
If we write $u$ for $u^+=u^-$, then  
\[
u^{\ep_k} \to u \ \ \text{ in }C(\ol U).
\] 
Moreover, as observed already, 
\[
u\in C^2(\ol U\setminus\pl V) \ \ \text{ and } \ \ 
\lim_{k\to\infty}u^{\ep_k} =u \ \ \text{ in } C^2(\ol U\setminus \pl V). 
\]
It is clear from \erf{vvf-2} and \erf{vvf-3} that $u$ satisfies \erf{balance} as well as  \erf{outside'} and \erf{dc}.

\smallskip

To complete the proof, it suffices  to show there is only one $u\in C(\ol U)\cap C^2(\ol U\setminus \pl V)$  satisfying 
\erf{outside'}--\erf{balance}. 
\smallskip

Assume that  $v, w\in C(\ol U)\cap C^2(\ol U\setminus \pl V)$ 
satisfy \erf{outside'}--\erf{balance}.  
Then,  as already discussed  above for $u$, 
the strong maximum principle yields that 
$v$ and $w$ are constant on $\ol V$. 
\smallskip

Set $\phi=v-w$ on $\ol U$ and note that, for some  $c\in \R$, 
\[
\phi=0 \  \text{ on  } \  \pl U \  \text{ and } \  \phi=c \ \text{ on }\ol V;
\]
interchanging $v$ and $w$ if needed, we may assume that $c\geq 0$. 
\smallskip

It follows from  \erf{outside'} and \erf{balance} that 
\[
-a_{ij}\left(\fr{\phi_{x_i}}{\gl}\right)_{x_j}-2b_i\phi_{x_i}=0 \ \ \text{ in } U\setminus\ol V 
\ \ \ \text{ and } \ \ \ 
\int_{\pl U}\fr{a_{ij}\phi_{x_i}\nu_j m}{\gl}d\gs=0.
\]
If $c=0$, 
then the maximum principle 
gives  $\phi = 0$ on $\ol U\setminus V$ and  $v=w$ on $\ol U$. 
\smallskip

If $c>0$, then the strong maximum principle 
implies that $\phi>0$ in $U\setminus \ol V$, moreover, Hopf's lemma yields 
\[
a_{ij}\phi_{x_i}\nu_j<0 \ \ \text{ on }\pl U,
\]
and, hence, 
\[
\int_{\pl U}\fr{a_{ij}\phi_{x_i}\nu_j m}{\gl}d\gs<0, 
\]
which is a contradiction. We thus conclude that $v=w$ on $\ol U$.
%
\eproof

Next we turn to the proof of Lemma \ref{lem:E}. For this we need an additional lemma. In preparation, for $\gd>0$, we write 
\[
W_\gd:=(\pl V)_{\gd}=\{x\in \ol U: \dist(x, \partial V) < \delta \}.
\] 

\begin{lem}\label{l0.1}There exists $\gd\in(0,\,1)$ and, for each $\ep\in[0,\,1)$,  $\psi^\ep\in C^2(\ol W_\gd)$ such that 
\beq\label{0.3}
a_{ij}\left(\fr{\psi^\ep_{x_i}}{\gl+\ep}\right)_{x_j}+2b_i\psi^\ep_{x_i}\leq 0  \ \text{ in } \ \ol W_\gd,  
\eeq
and, as $\ep \to 0$, $\psi^\ep \to \psi^0 \  \text{ in }  \ C^2(\ol W_\gd)$ with  $\psi^0\equiv0 \  \text{ in } \ \ol V\cap \ol W_\gd$ and 
$\psi^0>0 \ \text{ in } \  \ol W_\gd\setminus\ol V.$
\end{lem}

\bproof Let $\gd, K >0$ be such that $K\gd \leq \fr 12$ and $\ol {W}_\gd\subset U$,  
define, for $(\ep,x)\in [0,\,1]\tim \ol W_\gd$,
\[
\psi^\ep(x)=\int_0^{d(x)}(\gl_0(t)+\ep)(1-Kt)dt \ \ \ \text{ for }(\ep,x)\in [0,\,1]\tim \ol W_\gd,
\]
and note that, as $\ep\to 0$, 
\[ 
\psi^\ep(x)=\psi^0(x)+ \ep\int_0^{d(x)}(1-Kt)dt \to \psi^0(x) 
\ \ \text{ in }\ C^2(\ol W_\gd).
\]

Let $x\in \ol W_\gd$ and note that for $ t\in [-|d(x)|,\,|d(x)|]$, 
\[
1-Kt\geq 1-|Kt|\geq 1-K\gd\geq \fr 12. 
\]
Since $d>0$ in  $ \ol W_\gd\setminus \ol V$, we find 
\[
\psi^0(x)=\int_0^{d(x)}\gl_0(t)(1-Kt)dt\geq \fr 12 \int_0^{d(x)}\gl_0(t)dt>0,  
\]
while, since $d\leq 0$  in  $\ol W_\gd\cap \ol V$,  thanks to \eqref{ass8}, 
\[
\psi^0(x)=\int_0^{d(x)}\gl_0(t)(1-Kt)dt\equiv 0. 
\]
To show that \erf{0.3} holds,  fix  $\ep\in(0,\,1)$ and  compute 
\[
\psi^\ep_{x_i}=(\gl(x)+\ep)(1-Kd(x))d_{x_i},
\]
and 
\[
\left(\fr{\psi^\ep_{x_i}}{\gl+\ep}\right)_{x_j}=\left[(1-Kd)d_{x_i}\right]_{x_j}
=-Kd_{x_i}d_{x_j}+(1-Kd)d_{x_ix_j}. 
\]
Note that there exists  $C_1>0$,  which is independent of the choice of $K$,  such that 
\[
|(1-Kd)|\leq 1+\fr 12\leq 2 \ \text{ and} \ |a_{ij}(1-Kd)d_{x_ix_j}|\leq 2|a_{ij}d_{x_ix_j}|\leq C_1 \ \text{in} \ \ol W_\gd,
\] 
and 
\[
|b_i\psi^\ep_{x_i}|=|b_i(\gl+\ep)(1-Kd)d_{x_i}|
\leq 2 |b|(\gl+\ep)\leq C_1  \ \text{in} \ \ol W_\gd.
\]
Thus 
\[
a_{ij}\left(\fr{\psi^\ep_{x_i}}{\gl+\ep}\right)_{x_j}
+2b_i\psi^\ep_{x_i}\leq -K\gth|Dd|^2+ 3C_1=3C_1-K\gth \ \ \text{ for } x\in \ol W_\gd. 
\]
We fix $K>0$ so that 
$3C_1-K\gth\leq 0$
and  conclude that $\psi^\ep$ satisfies \erf{0.3}. 
\eproof

\bproof[Proof of Lemma \ref{lem:E}] To prove (i) we apply the maximum principle to  
$u^\ep-\|g\|_{C(\pl U)}$ and $-\|g\|_{C(\pl U)}-u^\ep$ and get 
\[
\sup_{ \ep\in (0,1)}\|u^\ep\|_{C(\ol U)}\leq \|g\|_{C(\pl U)}. 
\]

Next we show that $u^+$ is a viscosity subsolution of \erf{inside}.  Since 
\[
a_{ij}u^\ep_{x_ix_j}+\ep b_iu^\ep_{x_i}=0 \  \text{ in } \ V,
\]
it is well-known that $u^+$ is a viscosity subsolution to  $-a_{ij}w_{x_ix_j}=0$ in $V$. 
Thus the only issue  is to show that $u^+$ satisfies the boundary condition in the viscosity sense.   
\smallskip

Let $\phi\in C^2(\ol U)$ and assume that $x_0\in\pl V$ is a strict maximum point of $u^+-\phi$ on $\ol V$.
\smallskip

Arguing by contradiction, we suppose that
\[
-a_{ij}\phi_{x_ix_j}(x_0)>0 \ \ \text{ and } \ \ a_{ij}\phi_{x_i}\nu_j(x_0)>0.
\]
Let $\gd\in(0,\,1)$, $W_\gd$ and $\{\psi^\ep\}_{\ep\in[0,\,1)}$ be as in Lemma~\ref{l0.1}, select $\rho\in(0,\,\gd)$ 
so that 
\[a_{ij}\phi_{x_ix_j}<0 \ \ \text{ and } \ \ a_{ij}\phi_{x_i}d_{x_j}>0 \   \text{ in  } \  B:=B_\rho(x_0) \subset W_\delta.
\]
Since 
\[
a_{ij}\left(\fr{\phi_{x_i}}{\gl+\ep}\right)_{x_j}+2b_i\phi_{x_i}
=\fr{1}{\gl+\ep}\left(a_{ij}\phi_{x_ix_j}-\fr{\gl_0'(d)a_{ij}\phi_{x_i}d_{x_j}}{\lambda + \ep}+2(\gl+\ep)b_i\phi_{x_i}\right),
\]
we may choose $\ep_0\in(0,\,1)$ and reselect $\rho>0$ sufficiently small so that, for $\ep\in(0,\,\ep_0),$
\beq\label{0.4}
a_{ij}\left(\fr{\phi_{x_i}}{\gl+\ep}\right)_{x_j}+2b_i\phi_{x_i}<0 \ \text{ in }  \ \ol B. 
\eeq
For each $\ep\in (0,\,\ep_0)$, 
we select 
$x_{\ep}\in\ol B$ so that
\beq\label{0.5}
(u^\ep-(\phi+\psi^{\ep}))(x_{\ep})=\max_{\ol B}(u^\ep-(\phi+\psi^{\ep})). 
\eeq
In view of the definition of $u^+$, we may choose  $y_k \in B$ and $\ep_k \in (0,\,\ep_0)$ such that, as  $k\to \infty$, 
\[
y_k\to x_0,  \ \ep_k\to 0 \ \text{and} \ u^{\ep_k}(y_k)\to u^+(x_0).
\]
We may also assume that 
there is $\bar x\in\ol B$ such that
$\lim_{k\to\infty}x_{\ep_k}=\bar x.$ 
\smallskip

It follows from  \erf{0.5} that
\[
(u^{\ep_k}-(\phi+\psi^{\ep_k}))(x_{\ep_k})\geq (u^{\ep_k}-(\phi+\psi^{\ep_k}))(y_k), 
\]
and thus 
\beq\label{0.6}
(u^+-(\phi+\psi^{0}))(\bar x)\geq (u^+-(\phi+\psi^{0}))(x_0)=(u^+-\phi)(x_0),
\eeq
which implies $(u^+-\phi)(x_0)\leq (u^+-\phi)(\bar x)$, since 
$\psi^0\geq  0$ in $\ol B$, and that $\bar x=x_0$ because 
$x_0$ is a unique maximum point of $u^+-\phi$.

Selecting  $k\in\N$ large enough so that  $x_{\ep_k}\in B,$   
we deduce using  \erf{0.3}, \erf{0.4} and the maximum principle that, at $x_{\ep_k}$, 
\[\begin{aligned}
0&\leq a_{ij}\left(\fr{u^{\ep_k}_{x_i}}{\gl+\ep_k}\right)_{x_j}
+2b_i u^{\ep_k}_{x_i}
\\&\leq 
a_{ij}\left(\fr{(\phi+\psi^{\ep_k})_{x_i}}{\gl+\ep_k}\right)_{x_j}
+2b_i(\phi+\psi^{\ep_k})_{x_i}
\\&\leq 
a_{ij}\left(\fr{\phi_{x_i}}{\gl+\ep_k}\right)_{x_j}+2b_i\phi_{x_i}<0.
\end{aligned}
\] 
This is a contradiction and,  thus, $u^+$ is a viscosity subsolution of \erf{0.3}. 
\smallskip

The argument for the supersolution property is similar.
\eproof
\section{ The proof of Theorem \ref{thm:Ad}}


We remark that the existence of $m^\ep\in C^2(\ol U)$ that satisfies \ref{Ade}
follows from the following Fredholm  alternative type of argument. 
\smallskip

The adjoint problem to \ref{Ade} is the Neumann boundary value problem
\beq\label{Ade-ad}
\bcases \disp
a_{ij}\left(\fr{v_{x_i}}{\gl+\ep}\right)_{x_j}+2b_i v_{x_i}=0 \ \text{ in } \  U,&\\
a_{ij}v_{x_i}\nu_j=0 \ \text{ on } \ \pl U. 
\ecases
\eeq
Since any constant function is a solution to \eqref{Ade-ad},
the eigenvalue problem 
\beq\label{Ade-ad-eigen}
\bcases \disp
a_{ij}\left(\fr{v_{x_i}}{\gl+\ep}\right)_{x_j}+2b_i v_{x_i}+\rho v= 0\  \text{ in } \ U,&\\
a_{ij}v_{x_i}\nu_j=0 \  \text{ on } \ \pl U,
\ecases
\eeq 
has $\rho=0$ as its principal 
eigenvalue. Consequently, in principle, the problem 
\beq\label{Ade-eigen}
\bcases \disp
\left(\fr{(a_{ij}v)_{x_i}}{\gl+\ep}-2b_jv\right)_{x_j}+\rho v=0 \  \text{ in } \ U,&\\ \disp
\left(\fr{(a_{ij}v)_{x_i}}{\gl+\ep}-2b_jv\right)\nu_j=0 \ \text{ on } \ \pl U,
\ecases
\eeq
should have $\rho=0$ as its principal eigenvalue and there should be a positive 
function $m^\ep\in C^2(\ol U)$ that satisfies \ref{Ade}.  
\smallskip

We organize the important parts of the proof of Theorem \ref{thm:Ad} 
in a series of lemmata. 

\begin{lem}\label{exist-me}There exists a unique solution  
$m^\ep\in C^2(\ol U)$ of \emph{\ref{Ade}}. 
Moreover, If $\mu\in C^2(\ol U)$ satisfies the first two equations of 
\emph{\ref{Ade}}, then there exists $c\in\R$ such that $\mu=c m^\ep$ on $\ol U$.   
\end{lem}

We postpone the proof of the lemma above until the end of the proof of 
Theorem \ref{thm:Ad} and we continue with several other technical steps.
\smallskip

\begin{lem} \label{exist-psi0} There exists a positive solution  
$\psi_0\in C^2(\ol V)$ to \erf{Ad2}. 
Furthermore, if  $\phi\in C^2(\ol V)$ is a solution of  \erf{Ad2}, 
then  $\phi=c \psi_0$ on $\ol V$ for some $c\in \R$.
\end{lem}

\bproof The claim  is a consequence of Lemma \ref{exist-me}, with 
$U$, $\gl+\ep$ and $b_i$ replaced by $V$, $1$ and $0$, respectively.  
\eproof

\begin{lem} \label{unique-m} There exists at most one $m\in C(\ol U)\cap C^2(\ol U\setminus\pl V)$ that satisfies \erf{Ad1} and \erf{Ad2}. 
\end{lem}

We prepare the next result which is needed for the proof of the lemma above. 
For $\gamma>0$ we set
\[
V_\gamma=\{x\in\R^n\mid \dist(x,V)<\gamma\},
\]
which for sufficiently small $\gamma$ is a $C^{2,\ga}$-domain and $\ol V_\gamma \subset U$,  and consider the Dirichlet problem
\beq\label{Dir1}
\bcases \disp 
a_{ij}\left(\fr{v_{x_i}}{\gl}\right)_{x_j}+2b_jv_{x_j}=0 \ \ \text{ in } U\setminus \ol V_{\gamma}, &\\[3pt]
v=0 \ \ \ \text{ on }\pl V_{\gamma}\ \ \text{ and } \ \  
v=1 \ \ \ \text{ on }\pl U. 
\ecases
\eeq 
The classical Schauder theory (see \cite{GiTr}*{Theorem 6.14})
and  the hypotheses of Theorem~\ref{thm:Ad} yield that,  for $\gamma>0$ sufficiently small,  \eqref{Dir1} has  a unique solution $v^\gamma\in C^{2,\ga}(\ol U\setminus V_\gamma)$. 

\begin{lem}\label{lem:Dir}  
There exist constants  $\gamma_0\in(0,\,1)$ and $C>0$ such that,  if $\gamma\in (0,\,\gamma_0)$, then the Dirichlet problem 
\erf{Dir1} has a unique solution $v^\gamma\in C^{2,\ga}(\ol U\setminus V_\gamma)$ 
and it satisfies 
\[\bcases
|Dv^{\gamma}(x)|\leq C\gl(x) \ \ 
\text{ for all }x\in\pl V_{\gamma},&\\[3pt]
v^\gamma(x)\leq C\gL_0(d(x)) \ \ 
\text{ for all }x\in \ol U \setminus 
V_\gamma,
\ecases
\] 
where $\gL_0$ denotes the primitive of $\gl_0$ given by 
$
\gL_0(r):=\int_0^r\gl_0(t)dt. 
$
\end{lem}

\bproof Let $\gd\in (0,\,1)$ and $\psi^0\in C^2(\ol W_{\gd})$ 
be from Lemma~\ref{l0.1}. We may assume by replacing 
$\gd>0$ by a smaller number that, if $0<\gamma<\gd$, then $U\setminus \ol V_\gamma$ is 
a $C^{2,\ga}$-domain. The Schauder theory 
guarantees that, if $\gamma\in(0,\,\gd)$, 
there is a unique solution 
$v^\gamma\in C^{2,\ga}(\ol U\setminus V_{\gamma})$ of \erf{Dir1}.  

According to the proof of Lemma~\ref{l0.1}, the function 
$\psi^0$ has the form
\[
\psi^0(x)=\gL(d(x)) \ \ \text{ in } \ol V_{\gd_0}\setminus V,
\]
where $\gL\in C^3([0,\,\gd])$ satisfies the conditions  that 
$\gL(0)=0$, $\gL(r)>0$ for $r\in (0,\gd]$, 
and $\gL$ is nondecreasing on $[0,\,\gd]$. A careful 
review of the proof assures that 
$\gL'(r)\leq 2\gl_0(r)$ for $r\in[0,\,\gd]$ and, hence, 
$\gL(r)\leq 2\gL_0(r)$ for $r\in[0,\,\gd]$.  
Also, the function $\psi^0$ 
is a supersolution of   
\[
a_{ij}\left(\fr{v_{x_i}}{\gl}\right)_{x_j}+2b_jv_{x_j}=0 \ \ \text{ in } 
V_{\gd}\setminus \ol V. 
\]

Fix constants $\gamma_0\in(0,\,\gd)$ and $M>0$ 
so that $\gL(\gamma_0)<\gL(\gd)$ and  
$M(\gL(\gd)-\gL(\gamma_0))\geq 1$.

Let  $\gamma\in(0,\,\gamma_0)$, and consider the function 
\[
w(x):=M(\psi^0(x)-\gL(\gamma))=M(\gL(d(x))-\gL(\gamma)) \ \ \text{ on }
\ol V_{\gd}\setminus V_{\gamma}. 
\] 
Note that $w=0$ on $\pl V_{\gamma}$ and 
$w\geq 1$ on  $\pl V_{\gd}$. It is clear that $w$ is a supersolution 
of 
\[
a_{ij}\left(\fr{w_{x_i}}{\gl}\right)_{x_j}+2b_jw_{x_j}=0 \ \ \text{ in } 
V_\gd \setminus \ol V_{\gamma}.
\]
Since the constant functions $0$ and $1$ are a sub- and super-solution of 
\erf{Dir1} including the boundary conditions, we see by the maximum principle that $0\leq v^\gamma\leq 1$ on $\ol U\setminus V_\gamma$. 
Using again the maximum principle in the domain $V_\gd\setminus \ol V_\gamma$, we find that $v^\gamma\leq w$ on $\ol V_\gd\setminus V_\gamma$. Thus, we have $0\leq v^\gamma\leq w$ 
on $\ol V_\gd\setminus V_\gamma$, which yields 
\[\bcases
\text{$|Dv^\gamma|\leq M\gL'(\gamma)\leq 2M\gl_0(\gamma)$ on $\pl V_\gamma$,}&\\[3pt]
v^\gamma(x)\leq M\gL(d(x))\leq 2M\gL_0(d(x)) \ \ \text{ for } x\in V_\gd\setminus V_\gamma. 
\ecases 
\] 
The last inequality is valid even for $x\in\ol U\setminus V_\gd$, 
since $2M\gL_0(r)\geq 2M\gL_0(\gd)\geq 1$ for $r\geq \gd$.  
Thus, the lemma is valid with $C=2M$.
\eproof

\bproof[Proof of Lemma~\ref{unique-m}] Let $m_1,m_2\in C(\ol U)\cap C^2(\ol U\setminus \pl V)$  satisfy   
\erf{Ad1}--\erf{Ad2} and, as in Lemma~\ref{exist-psi0}, $\psi_0\in C^2(\ol V)$  a solution  to  \erf{Ad2} 
which is positive on $\ol V$.  Since $m_1,m_2>0$ on $\ol U$, we may 
choose a constant $c>0$ so that 
\[
\min_{\ol U}(cm_1-m_2)=0,
\]
and set $w=cm_1-m_2$ on $\ol U$.  

Lemma \ref{exist-psi0} yields  $\ga_1,\ga_2>0 $ such that 
\[
m_1=\ga_1\psi_0 \ \ \text{ and } \ \ m_2=\ga_2\psi_0 \ \ \ \text{ on }\ol V.
\]
Thus, $w=(c\ga_1-\ga_2)\psi_0$ on $\ol V$, which implies that either 
$w\equiv 0$ on $\ol V$ or $w>0$ on $\ol V$. 

We show that $w\equiv 0$ on $\ol U$. Consider first the case when $w$ has a minimum point at some point in $U\setminus \ol V$ 
and observe that, 
by the strong maximum principle,  $w\equiv 0$ in 
$\ol U\setminus V$, which implies $w\equiv 0$ on $\ol V$ 
as well.  Hence, $w\equiv 0$ on $\ol U$. 

Next, we assume that $w>0$ in $U\setminus\ol V$ and $w$ attains a 
minimum value $0$ at a point $x_0\in\pl U$. Hopf's lemma then 
gives that, at $x_0$, 
\[
\left(\fr{(a_{ij}w)_{x_i}}{\gl}-2b_jw\right)\nu_j=
\fr{a_{ij}w_{x_i}\nu_j}{\gl}<0,
\] 
which contradicts the second equality of \erf{Ad1}.  
  
What remains is the possibility where $w>0$ on $\ol U\setminus \ol V$ 
and $w\equiv 0$ on $\ol V$. 

Now, let $\gamma_0\in(0,\,1)$ and $C>0$ be the constants
from Lemma~\ref{lem:Dir}. According to the lemma, \erf{Dir1} 
has a solution $v^\gamma\in C^{2,\ga}(\ol U\setminus V_\gamma)$, 
$|Dv^\gamma|\leq C\gl $ on $\pl V_\gamma$,  
and 
\beq \label{uni-m1}
0\leq v^\gamma(x)\leq C\gL_0(d(x)) \ \ \text{ for } \ \ x\in\ol U\setminus V_\gamma,
\eeq 
 where the nonnegativity of $v^\gamma$ is a consequence of the maximum principle and $\gL_0$ is the primitive of $\gl_0$ chosen as in Lemma~\ref{lem:Dir}. 

By the Schauder estimates, for any compact $K\subset \ol U\setminus \ol V$, there exists $C_K>0$ such that, 
if $\gamma>0$ is sufficiently small, then  
$\|v^\gamma\|_{C^{2,\ga}(K)}\leq C_K$.  
Thus, we may choose a sequence $\{\gamma_k\}_{k\in\N}\subset (0,\,\gamma_0)$ converging to zero and a function $v^0\in 
C^2(\ol U\setminus \ol V)$ 
such that, for any compact $K\subset \ol U\setminus \ol V$, 
as $k\to \infty$, 
\[
v^{\gamma_k} \to v^0 \ \ \text{ in }C^2(K). 
\] 
Moreover, in view of \erf{uni-m1}, 
we may assume that $v^0\in C(\ol U\setminus V)$, 
$v^0=0$ on $\pl V$, and $v^0=1$ on $\pl U$

Applying  Green's formula 
\erf{vvf-1},  with $(\phi,\psi,\gl+\ep, U)$ replaced by $(v^\gamma,w,\gl, U\setminus \ol V_\gamma)$, we get
\beq\label{uni-m2}
0=\int_{\pl U}\fr{a_{ij}v_{x_i}^\gamma \nu_j w}{\gl}d\gs
-\int_{\pl V_\gamma}\fr{a_{ij}v_{x_i}^\gamma \nu_j w}{\gl}d\gs,
\eeq
where the unit normal vector $\nu$ on $\pl V_\gamma$ is taken 
as being exterior normal to $V_\gamma$. 
\smallskip

Note that, in view of by Lemma~\ref{lem:Dir} and,  for some independent of $\gamma$,  $C_1>0$,  
\[
\Big|\int_{\pl V_\gamma}\fr{a_{ij}v_{x_i}^\gamma \nu_j m}{\gl}d\gs
\Big| \leq C_1 \|m\|_{C(\pl V_\gamma)}\|\gl^{-1}Dv^\gamma\|_{C(\pl V_\gamma)}\leq CC_1\|m\|_{C(\pl V_\gamma)}.
\]

Setting $\gamma=\gamma_k$ and sending $k\to\infty$, 
we obtain from \erf{uni-m2} 
\beq\label{uni-m3}
\int_{\pl U}\fr{a_{ij}v^0_{x_i}\nu_j m }{\gl}d\gs=0. 
\eeq
It is obvious that $v^0\in C(\ol U\setminus V)\cap C^2(\ol U\setminus \ol V)$ solves \erf{Dir1}, with $V_\gamma$ replaced by $V$, and, for all $x\in \pl U$, $v_0(x)=1=\max_{\ol U \setminus V} v^0. $ By the strong maximum principle and 
Hopf's lemma, we deduce that 
\[
a_{ij}v_{x_i}^0\nu_j>0 \ \ \text{ on } \pl U. 
\]  
In our current situation, we have $m>0$ on $\pl U$, which together with the above inequalities gives a contradiction to \erf{uni-m3}, and   
thus we conclude that $w\equiv 0$ on $\ol U$. 

The third identity of \erf{Ad1} yields 
\[
0=\int_U wdx=c\int_U m_1dx-\int_U m_2dx=c-1,
\] 
from which we get $c=1$, and, thus, $m_1-m_2=w=0$ on $\ol U$. 
\eproof

\begin{lem} \label{m+-} For each $\ep\in(0,\,1)$, let $m^\ep$ be the unique solution to  \emph{\ref{Ade}}.
Assume that the family 
$\{m^{\ep_j}\}_{j\in\N}$ is uniformly bounded on $\ol U$, and let 
$m^\pm$ on $\ol U$ be the relaxed upper and lower limit of the $m^\ep$'s. 
Then $m^+$ and $m^-$ are respectively a viscosity sub- and super-solution to \erf{Ad2}  as functions on $\ol V$. 
\end{lem}

The proof of Lemma \ref{m+-} is very similar to the one of Lemma~\ref{lem:E}, with the role of Lemma~\ref{l0.1} 
replaced by Lemma~\ref{l0.2} below, hence we omit it.

 \begin{lem}\label{l0.2}There exist  $\gd\in(0,\,\gd_0)$ and, for each $\ep\in[0,\,1)$,   $\psi^\ep\in C^2(\ol W_\gd)$
such that 
\beq\label{0.11}
\left(\fr{(a_{ij}\psi^\ep)_{x_i}}{\gl+\ep}\right)_{x_j}-2(b_i\psi^\ep)_{x_i}\leq 0 \ \ \text{ in }\ol W_\gd 
\ \text{ if }\ \ep>0,
\eeq
and, as $\ep \to 0$, 
$\psi^\ep \to \psi^0 \ \text{ in } \ C^2(\ol W_\gd)$ with 
$ \psi^0\equiv 0 \  \text{ in } \ \ol V\cap \ol W_\gd$ and $\psi^0>0 \  \text{ in } \ \ol W_\gd\setminus\ol V.$
\end{lem}

The proof of the lemma above is similar to, but
slightly more involved than that of Lemma \ref{l0.1}, which needed 
the full strength of \eqref{ass8}. 

\bproof Let $\gd\in(0,\,\gd_0)$ and $K$ be positive constants such that $K\gd \leq \fr 12$ and $\ol {W}_\gd \subset U$. 
As in the proof of Lemma \ref{l0.1}, we define, for $(\ep,x)\in [0,\,1]\tim \ol W_\gd,. $
\[
\psi^\ep(x)=\int_0^{d(x)}(\gl_0(t)+\ep)(1-Kt)dt.  
\]
and note that the functions $\psi^\ep$ have all the claimed properties except \erf{0.11}. 
\smallskip 

To show  \erf{0.11}, we fix $\ep\in(0,\,1)$ and observe first that 
\[\begin{aligned}
\left(\fr{(a_{ij}\psi^\ep)_{x_i}}{\gl+\ep}\right)_{x_j}
&=\left(\fr{a_{ij}\psi^\ep_{x_i}+a_{ij,x_i} \psi^\ep}{\gl+\ep}\right)_{x_j}
\\&=\fr{a_{ij}\psi^\ep_{x_ix_j}+2a_{ij,x_i}\psi^\ep_{x_j}+a_{ij,x_ix_j}\psi^\ep}{\gl+\ep}
-\fr{a_{ij}\psi^\ep_{x_i}\gl_{x_j}+a_{ij,x_i}\psi^\ep\gl_{x_j}}{(\gl+\ep)^2}
\\&=a_{ij}\left(\fr{\psi^\ep_{x_i}}{\gl+\ep}\right)_{x_j}
+\fr{2a_{ij,x_i}\psi^\ep_{x_j}+a_{ij,x_ix_j}\psi^\ep}{\gl+\ep}
-\fr{a_{ij,x_i}\psi^\ep\gl_{x_j}}{(\gl+\ep)^2}.
\end{aligned}
\]
As seen  in the proof of Lemma~\ref{l0.1}, we have 
\[
|D\psi^\ep|=(\gl+\ep)(1-Kd(x))|Dd|\leq 2(\gl+\ep). 
\]
and, for some $C_1>0$,  
\[
a_{ij}\left(\fr{\psi^\ep_{x_i}}{\gl+\ep}\right)_{x_j}\leq C_1-K\gth \ \ \text{on} \ \  \ol W_\gd.
\]
Moreover, if $C_0$ is the constant from \erf{ass8}, 
\[
|\psi^\ep|\leq 2\Big|\int_0^{d(x)}(\gl_0(t)+\ep)dt\Big|
\leq 2(\gl+\ep)|d(x)|\leq 2\gd(\gl+\ep), 
\]
and 
\[\begin{aligned}
|\gl_{x_j}\psi^\ep|
\leq \gl_0'(d(x))|d_{x_j}||\psi^\ep|
\leq 2\gl_0'(d(x))|d(x)|(\gl+\ep)\leq 2C_0(\gl(x)+\ep)^2.
\end{aligned}\]

Hence, we can choose a constant $C_2>0$, which is independent of $K$ and $\ep$, such that  
\[
\left|\fr{2a_{ij,x_i}\psi^\ep_{x_j}+a_{ij,x_ix_j}\psi^\ep}{\gl+\ep}
-\fr{a_{ij,x_i}\psi^\ep\gl_{x_j}}{(\gl+\ep)^2}\right|\leq C_2.
\]
and 
\[
|(b_i\psi^\ep)_{x_i}|\leq |b_i||\psi_{x_i}^\ep|+|b_{i,x_i}\psi^\ep|\leq C_2. 
\]
It follows that 
\[
\left(\fr{(a_{ij}\psi^\ep)_{x_i}}{\gl+\ep}\right)_{x_j}-2(b_i\psi^\ep)_{x_i}\leq C_1+2C_2-\gth K \ \ 
\text{ on }\ol W_\gd.
\]
Choosing $K\geq (C_1+2C_2)/\gth$, we obtain \eqref{0.11}.
\eproof

\bproof[Proof of Theorem \ref{thm:Ad}] Assertion (i) is an immediate 
consequence of Lemma \ref{exist-me}. 
According to \cite{Li83}*{Lemma 3.1}, there exists 
 $C_1>0$, which is independent of $\ep\in(0,\,1)$,
such that 
\beq \label{sup-me}
\sup_{\ep\in (0,\,1)}\|m^\ep\|_{C(\ol U)} \leq C_1. 
\eeq
The interior Schauder estimates (\cite{GiTr}*{Corollary 6.3}) also imply that,
for each compact $K\subset U\setminus \pl V$, there 
exists 
$C_K>0$, again independent of $\ep$, such that
\beq\label{int-est-me}
\sup_{\ep\in (0,\,1)}\|m^\ep\|_{C^{2,\ga}(K)}\leq C_K. 
\eeq

We choose a smooth domain 
$W$ such that $\ol V \subset W$ and $\ol W\subset U$ 
and set 
\[
h^\ep=\left(\fr{(a_{ij} m^\ep)_{x_i}}{\gl+\ep}-2b_j m^\ep\right)\nu_j
\ \text{ on } \ \pl W,
\]
where $\nu$ denotes the inward unit normal vector of $W$. 
\smallskip

Observe that $w=m^\ep$ satisfies 
\[
\bcases\disp
\left(\fr{(a_{ij}w)_{x_i}}{\gl+\ep}-2b_jw\right)_{x_j}=0 \  \text{ in }  \ U\setminus \ol W,&\\[5pt] \disp
\left(\fr{(a_{ij}w)_{x_i}}{\gl+\ep}-2b_j w\right)\nu_j=0 \ 
\text{ on } \ \pl U,&\\[7pt] \disp
\left(\fr{(a_{ij}w)_{x_i}}{\gl+\ep}-2b_j w\right)\nu_j=h^\ep \  \text{ on } \ \pl W. 
\ecases
\]
We use the global Schauder estimates (\cite{GiTr}*{Theorem 6.30}) to find   
$C_W>0$, independent of $\ep$, such that 
\beq \label{b-est-me}
\sup_{\ep\in (0,\,1)}\|m^\ep\|_{C^{2,\ga}(\ol  U\setminus W)}\leq C_W. 
\eeq

Combining \erf{int-est-me} and \erf{b-est-me} shows that, 
for each compact $K\subset \ol U\setminus \pl V,$ 
there exists   
$C_K>0$, independent of $\ep$, such that
\beq\label{est-me}
\sup_{\ep\in (0,\,1)}\|m^\ep\|_{C^{2,\ga}(K)} \leq C_K. 
\eeq
We may then select a sequence $\ep_j\to 0$ such that, as $j\to\infty$ and for some $m\in C(\ol U)\cap C^2(\ol U\setminus \pl V)$,
\beq \label{conv-me}
m^{\ep_j} \to m \ \text{ in } \ C^{2}(\ol U\setminus \pl V).
\eeq 

Let $m^+$ and $m^-$ be the relaxed upper and lower limits of the $m^{\ep_j}$'s, which exist in view of \erf{sup-me}, 
and observe that 
$m=m^+=m^-$, as function on $\ol U\setminus \ol V$, is a solution to 
\beq\label{6}
\bcases\disp
\left(\fr {(a_{ij}m)_{x_i}}{\gl}-2b_jm\right)_{x_j}=0 \ \text{ in }  \ U\setminus \ol V,&\\[3pt] \disp
\left(\fr{(a_{ij}m)_{x_i}}{\gl}-2b_jm\right)\nu_j=0 \  \text{ on } \ \pl U,
\ecases
\eeq
while, in view of Lemma \ref{m+-}, $m^+$ and $m^-$, as functions on $\ol V$, are respectively a sub- and super-solution to 
\beq\label{7}
\bcases\disp 
(a_{ij}m)_{x_ix_j}=0  \ \text{ in } \ V,&\\[3pt] \disp
(a_{ij}m)_{x_i}\nu_j=0\  \text{ on } \ \pl V.
\ecases
\eeq
Let $\psi_0\in C^2(\ol V)$ be the positive solution  to  \erf{Ad2} given by  Lemma \ref{exist-psi0}.
Since $m^+\geq m^-\geq 0$, there are exist constants $c^\pm\geq 0$ such that  
\[
\max_{\ol V}(m^+-c^+\psi_0)=0 \ \ \text{ and } \ \ \min_{\ol V}(m^--c^-\psi_0)=0.  
\]
Using the strong maximum principle,  we find 
\[
m^+=c^+ \psi_0 \ \ \text{ and } \ \ m^-=c^-\psi_0 \  \text{ on } \ \ol V.
\]
Since $m^+=m^-$ in $\ol U\setminus\pl V$, we must have $c^+=c^-$ and, accordingly, 
$m^+=m^-=c^+\psi_0$ on $\ol V$ for some $c^+\geq 0$.  Thus, 
\[
m^+=m^- \  \text{ on } \ \ol U, 
\] 
and, therefore, if we set $m=m^+=m^-$ on $\pl V$, then
\beq\label{8}
\lim_{j\to\infty}m^{\ep_j}=m \  \text{ in } \ C(\ol U),
\eeq
which completes the proof of assertion (ii). 
\smallskip

Now, in view of Lemma  \ref{unique-m}, it only remains to show that 
$m$ is positive on $\ol U$. Note that $m\geq 0$   
and $m\not\equiv 0$ on $\ol U$.   Since $m$ satisfies \erf{6} and \erf{7},
we infer using again the strong maximum principle together with the Hopf's lemma  that,  if $m$ vanishes at a point 
in $\ol V$, then $m=0$ on $\ol V$, and that, if $m$ vanishes at a point in $\ol U\setminus \ol V$, then $m=0$ on $\ol U\setminus \ol V$.  In particular, 
if $m$ vanishes  at a point in $\ol U\setminus \ol V$, then $m=0$ on $\ol U$, which is impossible. That is, we must have $m>0$ in $\ol U\setminus \ol V$.  
\eproof
We conclude with the last remaining proof.
\bproof[Proof of Lemma \ref{exist-me}]
 Given $f\in C(\ol U)$, consider the problems 
\beq\label{Ade-ad-f}
\bcases \disp
-a_{ij}\left(\fr{v_{x_i}}{\gl+\ep}\right)_{x_j}-2b_i v_{x_i}=\rho v+f\  \text{ in } \ U,&\\
a_{ij}v_{x_i}\nu_j=0 \  \text{ on } \ \pl U,
\ecases
\eeq 
and
\beq\label{Ade-f}
\bcases \disp
-\left(\fr{(a_{ij}v)_{x_i}}{\gl+\ep}-2b_jv\right)_{x_j}=\rho v+f \ \text{ in } \ U,&\\ \disp
\left(\fr{(a_{ij}v)_{x_i}}{\gl+\ep}-2b_jv\right)\nu_j=0 \ \text{ on } \ \pl U.
\ecases
\eeq

The Schauder theory (\cite{GiTr}*{Theorem 6.31}) guarantees that, if $f\in C^{0,\ga}(\ol U)$ and $\rho<0$, then  \erf{Ade-ad-f} has a unique classical solution $v\in C^{2,\ga}(\ol U)$ 
and any constant function is a solution of \erf{Ade-ad-f}, for  $\rho=0$ and $f=0$. 
It follows that $\rho=0$ is the principal eigenvalue for the eigenvalue problem corresponding to \erf{Ade-ad-f}. 
\smallskip

For $r>0$, let $S_r$ be the solution operator to \erf{Ade-ad-f}, that is, for
$f\in C^{0,\ga}(\ol U)$,  $v=S_r f\in C^2(\ol U)$ is the unique solution to  \erf{Ade-ad-f} with $\rho=-r$. 
\smallskip

The maximum principle gives 
\[
\|S_rf\|_{C(\ol U)}\leq r^{-1}\|f\|_{C(\ol U)},
\]
which extends the domain of the $S_r$ to $C(\ol U)$. Obviously, for $f\in C(\ol U)$, $S_rf$ is 
the  unique viscosity solution to  \erf{Ade-ad-f}. 
\smallskip

The classical existence and uniqueness theory for elliptic equations does not immediately apply to  \erf{Ade-f}. 
In order to have a good monotonicity with respect to  the boundary conditions, we need to  make a change of  
unknowns. 
\smallskip

Let $\phi, v\in C^2(\ol U)$  and set
$w(x)=e^{-\phi(x)}v(x).$
Straightforward computations yield 
\[
\fr{(a_{ij} v)_{x_i}}{\gl+\ep}-2b_j v
=\fr{(a_{ij} e^{\phi}w)_{x_i}}{\gl+\ep}-2b_j e^{\phi}w
=e^{\phi}\left(\fr{a_{ij}w_{x_i}}{\gl+\ep}+ \left(\fr{a_{ij,x_i}+a_{ij}\phi_{x_i}}{\gl+\ep}-2b_j\right) w\right),
\]
and 
\[\begin{aligned}
\left(\fr{(a_{ij} v)_{x_i}}{\gl+\ep}-2b_j v\right)_{x_j}
&=\left\{
e^{\phi}\left(\fr{a_{ij}w_{x_i}}{\gl+\ep}+ \left(\fr{a_{ij,x_i}+a_{ij}\phi_{x_i}}{\gl+\ep}-2b_j\right) w\right)
\right\}_{x_j}
\\&=e^\phi\Bigg\{\fr{a_{ij}u_{x_ix_j}}{\gl+\ep}
+\left(\fr{2a_{ij}\phi_{x_j}+2a_{ij,x_j}}{\gl+\ep}
-\fr{a_{ij}\gl_{x_j}}{(\gl+\ep)^2}-2b_i
\right)w_{x_i}
\\&\quad +\left(\phi_{x_j} \left(\fr{a_{ij,x_i}+a_{ij}\phi_{x_i}}{\gl+\ep}-2b_j\right)
+ \left(\fr{a_{ij,x_i}+a_{ij}\phi_{x_i}}{\gl+\ep}-2b_j\right)_{x_j}\right)w
\Bigg\}.
\end{aligned}
\]
Choosing $\phi=M \dist(\cdot,\pl U)$ near $\pl U$, with $M>0$ sufficiently large, so that 
\[
D\phi=M\nu \  \text{ on }  \ \pl U,
\]
we may assume that 
\[
\left(\fr{a_{ij,x_i}+a_{ij}\phi_{x_i}}{\gl+\ep}-2b_j\right)\nu_j\geq 0 \ \text{ on  } \ \pl U. 
\]
Let $R>0$ be sufficiently large  so that 
\[
R\geq 1+\phi_{x_j} \left(\fr{a_{ij,x_i}+a_{ij}\phi_{x_i}}{\gl+\ep}-2b_j\right)
+ \left(\fr{a_{ij,x_i}+a_{ij}\phi_{x_i}}{\gl+\ep}-2b_j\right)_{x_j} \ \text{ on  } \ \ol U.
\]
If $v$ is a solution to \erf{Ade-f} with $\rho=-R$, then $w$ satisfies 
\beq\label{transf}
\bcases\disp
-\fr{a_{ij}w_{x_ix_j}}{\gl+\ep} 
-\tilde b_i w_{x_i}+(R-\tilde c) w=e^{-\phi}f \ \text{ in } \ U,&\\[2.5mm]\disp
\fr{a_{ij}w_{x_i}\nu_j}{\gl+\ep}+\tilde d w=0 \  \text{ on } \ \pl U, 
\ecases
\eeq
where 
\[\begin{aligned}
\tilde b_i(x)&=\left(\fr{2a_{ij}\phi_{x_j}+2a_{ij,x_j}}{\gl+\ep}
-\fr{a_{ij}\gl_{x_j}}{(\gl+\ep)^2}-2b_i
\right), 
\\\tilde c(x)&=\phi_{x_j} \left(\fr{a_{ij,x_i}+a_{ij}\phi_{x_i}}{\gl+\ep}-2b_j\right)
+ \left(\fr{a_{ij,x_i}+a_{ij}\phi_{x_i}}{\gl+\ep}-2b_j\right)_{x_j}, 
\\\tilde d(x)&=\left(\fr{a_{ij,x_i}+a_{ij}\phi_{x_i}}{\gl+\ep}-2b_j\right)\nu_j(x) .
\end{aligned}
\]
Note that 
\[
\tilde d\geq 0 \ \text{ on } \ \pl U \ \ \text{and} \ \ 
R-\tilde c\geq 1 \ \text{ on } \  \ol U.
\]
Applying again the maximum principle and the Schauder theory to \erf{transf}, we infer that, if $f\in C^{0,\ga}(\ol U)$, then  
\erf{Ade-f} has a unique classical solution $v\in C^{2,\ga}(\ol U)$ 
and satisfies the maximum principle. 
\smallskip

Let $T$ denote the solution operator for \erf{Ade-f} with $\rho=-R$, that is, if 
$v$ is a classical solution of \erf{Ade-f}, then $Tf=v$. 
\smallskip

As before applying the maximum principle, applied to the function $e^{-\phi}Tf$, we get
\[
\|e^{-\phi}Tf\|_{C(\ol U)}\leq \|e^{-\phi}f\|_{C(\ol U)}, 
\]
and, thus, 
\[
\|Tf\|_{C(\ol U)}\leq e^{2\|\phi\|_{C(\ol U)}}\|f\|_{C(\ol U)},
\]
which allows us to extend the domain of definition of $T$ to $C(\ol U)$. 
\smallskip

Fix  $r>0$ and observe that for any $\psi, f\in C(\ol U)$,
\beq\label{conti-linear}
\Big|\int_U\psi(x) S_r f (x)dx\Big|\leq |U|\,\|\psi\|_{C(\ol U)}
\|S_rf\|_{C(\ol U)}\leq r^{-1}|U|\,\|\psi\|_{C(\ol U)}\|f\|_{C(\ol U)},
\eeq
where $|U|$ denotes the Lebesgue measure of $U$, 
and, hence, that for each $\psi\in C(\ol U)$ the mapping
\[
C(\ol U)\ni f \mapsto \int_U\psi(x) S_r f (x)dx \in \R
\]
is linear and continuous.  Accordingly,  there exists \def\lan{\langle}\def\ran{\rangle}
a unique $S_r^*\psi\in C(\ol U)^*$, the dual space of $C(\ol U)$, such that, if $\lan \cdot, \cdot \ran$ denotes the duality pairing  between  $C(\ol U)^*$ and $C(\ol U)$, then, for all  $\psi\in C(\ol U),$
\[
\int_U\psi(x) S_r f (x)dx=\lan S_r^*\psi, f\ran.
\]
Using the Riesz representation theorem, we may identify $S_r^*\psi$ as a Radon measure on $\ol U$. 
\smallskip

Since, by \erf{conti-linear}, 
\[
|\lan S_r^\star \psi,f\ran|\leq r^{-1}|U|\,\|f\|_{C(\ol U)}\|\psi\|_{C(\ol U)} \ \ \text{ for }f, \psi\in C(\ol U),
\]
it follows that  $C(\ol U)\ni \psi \mapsto S_r^*\psi\in C(\ol U)^*$ is a continuous and linear map. 
\smallskip

Next, we fix $f\in C^{0,\ga}(\ol U)$ 
and $r\in (0,\,R)$, and solve \erf{Ade-f} for $\rho=-r$. 
\smallskip

Without loss of generality, we may assume that $f\geq 0$ in $\ol U$. 
\smallskip

We use an iteration argument, and consider the  sequence $\{v_n\}_{n\in \N}$ given by 
$v_1\equiv 0$ and, for $n> 1$, by the solution $v_n\in C^2(\ol U)$ of
\beq\label{Ade-f-n}
\bcases \disp
-\left(\fr{(a_{ij}v_{n})_{x_i}}{\gl+\ep}-2b_jv_{n}\right)_{x_j}=-Rv_{n} +(R-r)v_{n-1}+f \ \text{ in } \ U,&\\ \disp
\left(\fr{(a_{ij}v_{n})_{x_i}}{\gl+\ep}-2b_jv_{n}\right)\nu_j=0 \ \text{ on } \ \pl U,
\ecases
\eeq
Using  the operator $T$,  \erf{Ade-f-n} can be stated as
\beq\label{iteration}
v_{n}=T((R-r)v_{n-1}+f). 
\eeq
It follows from the maximum principle that, for all $n\in \N$, 
\beq\label{ineq-iteration}
v_n\geq 0 
\  \text{ and } \  v_{n+1}\geq v_n \  \text{ on } \ \ol U. 
\eeq

We show that, in the sense of measures on $\ol U$ and for all $n\in\N,$
\[
v_n\leq S_r^* f. 
\]
Indeed,  first observe that, in view of \erf{vvf-1},  for any $\phi, \psi\in C^2(\ol U)$, if 
\[
\left(\fr{(a_{ij}\phi)_{x_i}}{\gl+\ep}-2b_j\phi\right)\nu_j=a_{ij}\psi_{x_i}\nu_j=0 \ \text{ on } \ \pl U,
\]
then 
\[
\int_U \phi L_r\psi dx=\int_U L_r^*\phi\psi dx,
\]
where 
\[
L_r\psi=-a_{ij}\left(\fr{v_{x_i}}{\gl+\ep}\right)_{x_i} -2b_j\psi_{x_j} +r\psi 
\ \ \text{ and} \ \ 
L_r^*\phi=-\left(\fr{(a_{ij}\phi)_{x_i}}{\gl+\ep}-2b_j \phi\right)_{x_j}+r\phi. 
\]
We rewrite the above formula as 
\beq\label{duality}
\lan\phi,L_r\psi\ran=\lan L_r^*\phi,\psi \ran
\eeq
and apply it to  $(\phi,\psi)=(v_{n+1}, S_rw)$, with $n\in \N$ and $w\in C^{0,\ga}(\ol U)$, 
to get 
\[
\lan L_r^* v_{n+1}, S_rw\ran=\lan v_{n+1}, L_rS_rw\ran=\lan v_{n+1}, w\ran,
\]
where the first term can be calculated as follows:
\[
\lan  L_r^* v_{n+1}, S_rw\ran
=\lan L_R^* v_{n+1}+(r-R)v_{n+1}, S_r w\ran
=\lan (R-r)(v_{n}-v_{n+1})+f, S_r w\ran. 
\]
Assume now that $w\geq 0$ on $\ol U$, and observe that, by the maximum principle, 
$S_rw\geq 0$ on $\ol U$ and 
\[
\lan (R-r)(v_{n}-v_{n+1})+f, S_r w\ran\leq \lan f, S_rw\ran=\lan S_r^* f, w\ran.
\] 
Hence, if  $ w\geq 0,$
\[
\lan v_{n+1}-S_r^*f, w\ran\leq 0, 
\]
which proves that $v_{n+1}\leq S_r^*f$. 
\smallskip

In particular, for all $n\in \N,$ we have 
\beq \label{L1bound}
\int_U v_{n+1} dx \leq \lan S_r^* f, 1\ran \leq r^{-1} |U|\|f\|. 
\eeq 
It follows from  \cite{li83}*{Lemma 3.1} and \cite{GiTr}*{Theorem 6.30} that 
\[
\sup_{n\in\N}\|v_{n}\|_{C^{2,\ga}(\ol U)}<\infty, 
\]
and, hence, for some $v\in C^{2,\ga}(\ol U)$,
\[
\lim_{n\to\infty}v_n = v \ \text{ in } \ C^2(\ol U).
\]
Moreover, it is easily seen that 
 $v$ is a solution to  \erf{Ade-f} with $\rho=-r$. 
 \smallskip
 
 Also, 
using \erf{duality}, we deduce that, if $v\in C^2(\ol U)$ is a solution to  \erf{Ade-f}, with $\rho=-r$, then, 
for $w\in C^{0,\ga}(\ol U)$, 
\[
\lan S_r^* f, w \ran=\lan f, S_r w\ran
=\lan L_r^*v, S_r w\ran
=\lan v, L_r S_r w \ran
=\lan v,w\ran,
\]
which shows that $v=S_r^*f$ and, in particular, the uniqueness of the  solution to  \erf{Ade-f} for  $\rho=-r$. 
\smallskip

Fix a sequence $(0,\,R)\ni r_k \to 0$ and set 
$m_k=|U|^{-1}S_{r_k}^* r_k,$
where the last $r_k$ denotes the constant function $\ol U\ni x\mapsto r_k$.  
\smallskip
 
Note that $m_k\in C^2(\ol U)$ is a solution to \erf{Ade-f} with $\rho=-r_k$ and $f=|U|^{-1}r_k$, $m_k\geq 0$ on $\ol U$, and 
\[
\lan m_k, 1\ran=|U|^{-1}\lan r_k, S_{r_k}1\ran=|U|^{-1}\lan r_k, r_k^{-1}\ran=1. 
\]   
The Schauder theory (see \cite{li83}*{Lemma 3.1} and \cite{GiTr}*{Theorem 6.30}) imply that  
$\{m_k\}\subset C^{2,\ga}(\ol U)$ is bounded.  
\smallskip

Hence, after passing 
to a subsequence, we may assume that, for some $m^\ep\in C^{2,\ga}(\ol U)$, 
\[
\lim_{k\to\infty} m_k=m^\ep \ \text{ in } \ C^2(\ol U). 
\]
It follows immediately that $m^\ep$ is a solution of \ref{Ade}, except 
the positivity of $m^\ep$. 
%
It is 
clear that $m^\ep\geq 0$ and $m^\ep\not \equiv 0$. The strong maximum principle and  Hopf's lemma  
yield that $m^\ep>0$ on $\ol U$.  Hence, $m^\ep$ is a solution of 
\ref{Ade}. 
\smallskip

Let $\mu\in C^2(\ol U)$ satisfy the first two equations of \ref{Ade} and observe that, by the same reasoning as above,  
$\mu>0$ on  $\ol U$. 
%
Choose $c\in\R $ so that $\mu\leq cm^\ep$ on $\ol U$ 
and $\mu(x_0)=cm^\ep(x_0)$ for some $x_0\in\ol U$. 
%
Applying  the strong maximum principle and Hopf's lemma   to $c m^\ep-\mu$, we find that, 
if $\mu\not\equiv cm^\ep$ on $\ol U$, then
$\mu<c m^\ep$ in  $\ol U$, which is a contradiction.
It follows  that $\mu=cm^\ep$.  This also implies the uniqueness of a solution of \ref{Ade} and 
the proof is complete.    
\eproof



\section{The initial value problem \erf{takis5}}

In this section we briefly sketch the proof of  the following theorem.  

\begin{thm} \label{ce} Let $\mu>0$, $T>0$ and 
assume \erf{ass1} and \erf{ass2}. 

\emph{(i)} If $v, w$ are  bounded, respectively  upper and lower semicontinuous on $\R^{2n}\tim[0,\,T),$ viscosity sub- and super-solutions 
to  $u_t=L^\mu u$ in $\R^{2n}\in (0,\,T)$ and  $v(\cdot,\cdot,0)\leq w(\cdot,\cdot,0) $ in $\R^{2n},$
then, $v\leq w$ on 
$\R^{2n}\tim[0,\,T)$. 

\emph{(ii)} Let $u_0^\mu\in\BUC(\R^{2n})$. There exists a unique viscosity solution 
$u^\mu\in\bc(\R^{2n}\tim[0,\,T))$ to  \erf{takis5}. 
\end{thm}

The uniform continuity assumption on $u_0^\mu$ can be relaxed and replaced by the continuity of $u_0^\mu$ in the above theorem. 
But this strong assumption makes it easy to prove assertion (ii). 
\smallskip

Let $\lan x\ran:=(|x|^2+1)^{1/2}$ and set 
\begin{equation}\label{takis199}
p(x,y):=\lan x\ran+\fr 12 |y|^2 \ \ \text{ for }(x,y)\in\R^{n}\tim\R^n.
\end{equation}

Using   
the positivity of $\gl$ (see  \erf{ass2}) straightforward calculations imply that there exist exist $c>0$ and $C>0$ such that, for   all $(x,y)\in\R^{2n},$
\begin{equation}\label{takis200}
L^\mu p(x,y)\leq C-c|y|^2. 
\end{equation}

\bproof[Proof of Theorem \ref{ce}] 
Let $p$ and  $C, c$ as in \eqref{takis199} and \eqref{takis200}, 
set 
$q(x,y,t)=p(x,y)+Ct$ for $(x,y,t)\in\R^{2n}\tim[0,\,T]$, 
and note that, for any $(x,y,t)\in\R^{2n}\tim(0,\,T)$,
\beq\label{ce2}
L^\mu q(x,y,t)=L^\mu p(x,y)\leq C-c|y|^2=q_t(x,y,t)-c|y|^2. 
\eeq

To prove (i), 
we fix  $\gd>0$ and observe that 
\[
v_\gd(x,y,t)=v(x,y,t)-\gd q(x,y,t), 
\] 
 is an  upper semicontinuous 
 subsolution to $u_t=L^\mu u$ in 
$\R^{2n}\tim(0,\,T)$.  
\smallskip 

Since $v$ and $w$ are bounded and 
\[
p(x,y) \to \infty  \ \ \text{ as } \ \  |x|+|y|\to \infty,
\]
we can choose a bounded open subset $\gO$ of $\R^{2n}$ so that 
\beq\label{ce3}
v_\gd(x,y,t)\leq w(x,y,t) \ \ \text{ if } \ \ (x,y)\not\in \gO,
\eeq
while  
\[
v_\gd(x,y,0)\leq v(x,y,0)\leq w(x,y,0) \ \ \text{ for all } \ \ (x,y)\in\R^{2n}.
\]
Applying the  standard comparison theorem (for instance, 
\cite{CIL}*{Theorem 8.2} and its proof), we find that 
$v_\gd\leq w$ on $\gO\tim[0,\,T)$, which 
together with \erf{ce3} yields that $v_\gd\leq w$ on $\R^{2n}\tim[0,\,T)$. 
Letting  $\gd\to 0$ implies the claim.
\smallskip

We turn to (ii).  Since $u_0^\mu\in\BUC(\R^{2n})$, for each $\gd\in (0,\,1)$, we may choose $u_0^{\mu,\gd}
\in \bc^2(\R^{2n})$ so that $\|u_0^{\mu,\gd}-u_0^\mu\|<\gd$. 
Obviously, there exists a constant $C_\gd>0$ so that 
\[
|L^\mu u_0^{\mu,\gd}|\leq C_\gd(1+|y|) \ \ \text{ on } \ \ \R^{2n}. 
\]
Select  $M_\gd>0$ so that
\[
|L^\mu u_0^{\mu,\gd}|+\gd L^\mu p\leq 
C_\gd(1+|y|)+\gd(C-c|y|^2)\leq M_\gd \ \ \text{ on } \ \ \R^{2n},
\] 
set
\[
v_\gd^\pm(x,y,t)=u_0^{\mu,\gd}(x,y)\pm (\gd+\gd p(x,y)+M_\gd t)
\ \ \text{ on } \ \\R^{2n}\tim[0,\,T], 
\]

observe that $v_\gd^-, v_\gd^+\in C^2(\R^{2n}\tim[0,\,T])$ 
are a sub- and super-solution of $u_t=L^\mu u$ in 
$\R^{2n}\tim (0,\,T)$, and, finally, for $(x,y,t)\in\R^{2n}\tim[0,\,T]$, 
\beq\label{ce6}
\bcases
v_\gd^\pm(x,y,0)=u_0^{\mu,\gd}(x,y)\pm\gd(1+p(x,y)),&\\[2pt]
v_\gd^-(x,y,t)\leq u_0^\mu(x,y)\leq  v_\gd^+(x,y,t).
\ecases
\eeq
The stability property of viscosity 
solutions yields that, if, in  $\R^{2n}\tim[0,\,T]$, 
\[
v^+ =\inf_{0<\gd<1} v_\gd^+  \ \text{ and   } \ \  v^- =\sup_{0<\gd<1} v_\gd^+,  
\]
then the upper and lower semicontinuous envelopes  
$w^-$ of $v^-$  and $w^+$ of $v^+$ are respectively  a 
viscosity sub- and super-solution of 
$u_t=L^\mu u$ in $\R^{2n}\tim (0,\,T)$. 
Moreover, it follows from \erf{ce6}  that,   
for $(x,y,t)\in\R^{2n}\tim[0,\,T]$, 
\beq\label{ce7}\bcases
v^\pm(x,y,0)=u_0^\mu(x,y), &\\[2pt]
v^-(x,y,t)\leq w^-(x,y,t)\leq u_0^\mu(x,y)\leq w^-(x,y,t) \leq v^+(x,y,t). &
\ecases
\eeq
According to Perron's method (\cite{CIL}), if we set
\[\begin{aligned}
u^\mu(x,y,t)=\sup\{u(x,y,t)\mid &u \text{ is a subsolution of } \ u_t=L^\mu u \ \text{ in } \ \R^{2n}\tim(0,\,T),\, 
\\&w^-\leq u\leq w^+ \ \text{ on } \R^{2n}\tim[0,\,T)\}, 
\end{aligned}\]
then $u^\mu$ is a solution to  $u_t=L^\mu u$ in $\R^{2n}\tim(0,\,T)$ in the sense that the upper and lower semicontinuous envelopes $(u^{\mu})^*$ and $(u^\mu)_*$ of $u^\mu$ are respectively a sub- and super-solution to  $u_t=L^\mu u$ in 
$\R^{2n}\tim(0,\,T)$.  
\smallskip
 
Note that $v^-,-v^+$ are lower semicontinuous on $\R^{2n}\tim[0,\,T]$, and hence, by \erf{ce7}, $(u^\mu)^*(x,y,0)=(u^\mu)_*(x,y,0)
=u_0^\mu(x,y)$ for all $(x,y)\in\R^{2n}$. 
\smallskip

Thus, by the comparison assertion (i), we obtain $(u^{\mu})^*\leq 
(u^\mu)_*$ on $\R^{2n}\tim[0,\,T)$, 
and we  
conclude that $u^\mu \in \bc(\R^{2n}\tim[0,\,T))$ 
and it is a solution of \erf{takis5}.  
\smallskip

The uniqueness of $u^\mu$ is an 
immediate consequence of (i). 
\eproof

\end{document}